\documentclass[a4paper,10pt]{amsart}

\usepackage{amssymb}
\usepackage[arrow]{xy}
\usepackage{pb-diagram,pb-xy}
\usepackage{graphicx}
\usepackage[a4paper,dvips,bookmarks]{hyperref}

\theoremstyle{plain}
\newtheorem{theorem}{Theorem}[section]
\newtheorem{proposition}[theorem]{Proposition}
\newtheorem{corollary}[theorem]{Corollary}
\newtheorem{lemma}[theorem]{Lemma}

\newtheorem{claim}{Claim}
\theoremstyle{definition}
\newtheorem{definition}[theorem]{Definition}

\newtheorem{remark}[theorem]{Remark}

\def\Z{\mathbb{Z}}
\def\Q{\mathbb{Q}}
\def\R{\mathbb{R}}
\def\C{\mathbb{C}}

\def\sR{\mathcal{R}}

\def\M{\mathcal{M}}
\def\T{\mathcal{T}}
\def\hF{\widehat{F}}
\def\CL{\mathcal{C}^\mathrm{L}}
\def\CSL{\mathcal{C}^\mathrm{SL}}
\def\hCSL{\widehat{\mathcal{C}}^\mathrm{\,SL}}
\def\FL{\mathcal{F}^\mathrm{L}}
\def\FSL{\mathcal{F}^\mathrm{SL}}
\def\hFSL{\widehat{\mathcal{F}}^\mathrm{SL}}
\def\B{\mathcal{B}}
\def\BFSL{\mathcal{BF}^\mathrm{SL}}

\def\Coker{\operatorname{Coker}}
\def\Im{\operatorname{Im}}
\def\Hom{\operatorname{Hom}}
\def\Ext{\operatorname{Ext}}
\def\sign{\operatorname{sign}}
\def\rank{\operatorname{rank}}
\def\dis{\operatorname{dis}}

\def\to{\mathchoice{\longrightarrow}{\rightarrow}{\rightarrow}{\rightarrow}}
\makeatletter
\newcommand{\shortxra}[2][]{\ext@arrow 0359\rightarrowfill@{#1}{#2}}
\def\longrightarrowfill@{\arrowfill@\relbar\relbar\longrightarrow}
\newcommand{\longxra}[2][]{\ext@arrow 0359\longrightarrowfill@{#1}{#2}}
\renewcommand{\xrightarrow}[2][]{\mathchoice{\longxra[#1]{#2}}%
  {\shortxra[#1]{#2}}{\shortxra[#1]{#2}}{\shortxra[#1]{#2}}}
\makeatother

\makeatletter
\def\Nopagebreak{\@nobreaktrue\nopagebreak}
\makeatother

\begin{document}

\title%
[Structure of the string link concordance group]%
{Structure of the string link concordance group and Hirzebruch-type
  invariants}

\author{Jae Choon Cha}

\address{Department of Mathematics, Pohang University of Science and Technology\\
  Pohang, Kyungbuk 790--784, Republic of Korea}

\email{jccha@postech.ac.kr}

\def\subjclassname{\textup{2000} Mathematics Subject Classification}
\expandafter\let\csname subjclassname@1991\endcsname=\subjclassname
\expandafter\let\csname subjclassname@2000\endcsname=\subjclassname


\begin{abstract}
  We employ Hirzebruch-type invariants obtained from iterated
  $p$-covers to investigate concordance of links and string links.  We
  show that the invariants naturally give various group homomorphisms
  of the string link concordance group into $L$-groups over number
  fields.  We also obtain homomorphisms of successive quotients of the
  Cochran-Orr-Teichner filtration.  As an application we show that the
  kernel of Harvey's $\rho_n$-invariant is large enough to contain a
  subgroup with infinite rank abelianization, modulo local knots.  As
  another application, we show that recently discovered nontrivial
  2-torsion examples of iterated Bing doubles lying at an arbitrary
  depth of the Cochran-Orr-Teichner filtration are independent over
  $\Z_2$ as links, in an appropriate sense.  We also construct similar
  examples of infinite order links which are independent over~$\Z$.
\end{abstract}

\maketitle

%\begin{center}
%Not for general distribution
%\end{center}
%
%\bigskip

%\tableofcontents

\section{Introduction and main results}

In~\cite{Cha:2007-1}, the author defined new Hirzebruch-type
intersection form defect invariants from towers of iterated
$p$-covers, and gave applications to link concordance and homology
cobordism.  The aim of this paper is to employ the invariants to
reveal further information on the structure of the string link
concordance group and the set of link concordance classes.  In this
paper manifolds are assumed to be topological and oriented, and
submanifolds are locally flat.  All results hold in the smooth
category as well.

\subsection{Hirzebruch-type invariants}

We start with a quick review of the invariants defined
in~\cite{Cha:2007-1}.  Throughout this paper, $p$ denotes a prime.  For
a CW-complex $X$, a~tower
\[
X_n \to \cdots \to X_1 \to X_0=X
\]
consisting of connected covers $X_i$ of $X$ is called a
\emph{$p$-tower of height $n$ for $X$} if each $X_i\to X_{i-1}$ is a
regular cover whose covering transformation group is a finite abelian
$p$-group.  Suppose $M$ is a closed 3-manifold.  For a $p$-tower
$\{M_i\}$ of height $n$ and a character $\phi\colon \pi_1(M_n) \to
\Z_d$ with $d$ a power of $p$ such that $(M_n,\phi)=0$ in the
topological bordism group $\Omega^{top}_3(B\Z_d)$, an invariant
\[
\lambda(M_n,\phi) \in L^0(\Q(\zeta_d))
\]
is defined, where $\zeta_d=\exp(2\pi\sqrt{-1}/d)$ is the $d$th
primitive root of unity and $L^0(\Q(\zeta_d))$ is the group of Witt
classes of nonsingular hermitian forms over $\Q(\zeta_d)$.  (As usual,
$\Q(\zeta_d)$ is endowed with the involution
$\bar\zeta^{\vphantom{-1}}_d=\zeta_d^{-1}$.)  Basically
$\lambda(M_n,\phi)$ is the difference of the Witt class of the
$\Q(\zeta_d)$-coefficient intersection form of a 4-manifold $W$
bounded by $M_n$ over $\Z_d$ and that of the untwisted intersection
form of~$W$.  (For more details the reader is referred
to~\cite{Cha:2007-1}.)

We remark that in order to define the invariant it suffices to specify
$\phi$ as a map of a subgroup of $\pi_1(M)$, without considering the
$p$-tower~$\{M_i\}$: we call $\phi\colon H \to \Z_d$ a ($\Z_d$-valued)
\emph{$p$-virtual character} of a group $G$ if $\phi$ is a group
homomorphism of a subgroup $H$ in $G$ such that the index $[G:H]$ is
finite and a power of~$p$.  For any $p$-virtual character $\phi$ of
$\pi_1(M)$, it can be seen by an easy induction that there is a
$p$-tower $\{M_i\}$ of $M$ whose highest term $M_n$ is the cover
$M_\phi$ determined by $H$, i.e., the image of $\pi_1(M_\phi)$ under
the map induced by the covering projection is exactly $H \subset
\pi_1(M)$.  The invariant $\lambda(M_\phi,\phi)$ is independent of the
choice of~$\{M_i\}$ (provided it is defined).  Since in this paper we
always define a $p$-virtual character by constructing a $p$-tower, we
regard a $p$-virtual character of $\pi_1(M)$ as endowed with a
$p$-tower of $M$.  In other words, we think of a pair $(\{M_i\},\phi)$
of a $p$-tower $\{M_i\}$ and $\phi\colon \pi_1(M_n) \to \Z_d$.  We
call such a pair a \emph{$\Z_d$-valued $p$-structure of height $n$}
for~$M$.  (The order $d$ is always assumed to be a power of~$p$.)

\subsection{Invariants of the string link concordance group}

In \cite{Cha:2007-1}, it was shown that the Hirzebruch-type invariant
defined from a $p$-structure of the surgery manifold of a link in
$S^3$ gives an obstruction to being a slice link.  In this paper we
also think of \emph{string links} in the sense of Habegger and
Lin~\cite{Habegger-Lin:1990-1,Habegger-Lin:1998-1}, which has the
advantage that there is a well-defined group structure on the set of
concordance classes; we denote the group of concordance classes of
$m$-component string links by $\CSL(m)$, or simply~$\CSL$.

The invariants in \cite{Cha:2007-1} give rise to invariants of $\CSL$
with values in $L^0(\Q(\zeta_d))$ via the closures of string links.
As the first main result of this paper, we show that the invariants
restricted to a large class of string links become \emph{group
  homomorphisms}; we call a string link $\beta$ an \emph{$\widehat
  F$-string link} if the closure of $\beta$ is an $\widehat F$-link in
the sense of Levine~\cite{Levine:1989-1}.  Roughly speaking, it means
that the fundamental group of the complement of the closure of $\beta$
has Levine's algebraic closure~\cite{Levine:1989-1} identical with
that of a free group (and the preferred longitudes are trivial in the
algebraic closure).  We note that $\widehat F$-(string) links form the
largest known class of (string) links with vanishing Milnor's
$\bar\mu$-invariant; it is a big open problem in link theory whether
all (string) links with vanishing Milnor's $\bar\mu$-invariants are
$\hF$-(string) links.

Let $\hCSL=\hCSL(m)$ be the subgroup of $\CSL$ consisting of (the
classes of) $\hF$-string links.  An advantage of $\hF$-links is that
$p$-structures of the surgery manifolds are naturally described in
terms of those of a fixed space $X=\bigvee^m S^1$, the wedge of $m$
circles.  Namely, in Section~\ref{section:invariant-of-string-links},
we show that every $p$-structure $\T = (\{X_i\},\theta)$ of height $n$
for $X$ \emph{canonically} determines a $p$-structure $(\{M_i\},\phi)$
of height $n$ for the surgery manifold $M$ of the closure of an
$\hF$-string link~$\beta$, and any $p$-structure for $M$ arises in
this way.  We define $\lambda_\T(\beta) \in L^0(\Q(\zeta_d))$ to be
the Hirzebruch-type invariant $\lambda(M_n,\phi)$.  The result stated
below essentially says that $\lambda_\T(\beta)$ is additive under
product of $\hF$-string links.

\begin{theorem}
  \label{theorem:lambda_T-homomorophism}
  For any $\Z_d$-valued $p$-structure $\T=(\{X_i\},\theta)$ for $X$,
  the invariant $\lambda_\T(\beta)$ is well-defined for any
  $\hF$-string link, and this gives rise to a group homomorphism
  \[
  \lambda_\T\colon \hCSL \to L^0(\Q(\zeta_d)).
  \]
\end{theorem}

We remark that $\CSL$ and $\hCSL$ contain the knot concordance group
as a summand; $\CSL(1)=\hCSL(1)$ is isomorphic to the knot
concordance group and there is an obvious split injection
$\bigoplus^m\CSL(1) \to \hCSL(m) \subset \CSL(m)$ whose image is the
(normal) subgroup generated by ``local knots''.  (A precise definition
of a local knot is given in Section~\ref{section:local-knots}.)  In
order to study the sophistication peculiar to links, modulo knot
concordance, we consider the quotient of the string link concordance
group modulo local knots.  A subclass of our invariants can be used to
investigate this.  Let $x_i$ be the loop in $X=\bigvee^m S^1$
representing the $i$th circle.  We say that a $p$-structure
$\T=(\{X_i\},\theta)$ of height $n$ for $X$ is \emph{locally trivial}
if any element in $\pi_1(X_n)$ which projects to a conjugate of a
power of $[x_i]$ in $\pi_1(X)$ is in the kernel of~$\theta$.

\theoremstyle{plain}
\newtheorem*{tempthm}
{Addendum to Theorem~\ref{theorem:lambda_T-homomorophism}}

\begin{tempthm}
  If $\T$ is locally trivial, then $\lambda_\T$ induces a group
  homomorphism
  \[
  \lambda_\T\colon \frac{\hCSL}{\langle \text{local knots} \rangle} \to
  L^0(\Q(\zeta_d))
  \]
  where $\langle\text{local knots}\rangle$ denotes the subgroup
  generated by local knots.
\end{tempthm}

In fact our results on $\hF$-string links also hold for a
(potentially) larger class of string links, namely
\emph{$\Z_{(p)}$-coefficient $\hF$-string links} in the sense
of~\cite{Cha:2007-1}, where $\Z_{(p)}$ designates the localization of
$\Z$ away from~$p$; the definition of a $\smash{\Z_{(p)}}$-coefficient
$\hF$-string link is identical with that of a $\hF$-string link except
that the algebraic closure with respect to
$\smash{\Z_{(p)}}$-coefficients (see~\cite{Cha:2004-1}) is used in
place of Levine's algebraic closure.  Since there is a natural
transformation from Levine's algebraic closure to the algebraic
closure with $\Z_{(p)}$-coefficients, an $\hF$-(string) link is a
$\Z_{(p)}$-coefficient $\hF$-(string) link.  In all results in this
paper, the term ``$\hF$-(string) link'' can be understood as this
$\Z_{(p)}$-coefficient analogue.

\subsection{Structure of Cochran-Orr-Teichner filtration}

We use the homomorphism $\lambda_\T$ to investigate the structure of
the Cochran-Orr-Teichner filtration
\[
\cdots \subset \FSL_{(n.5)} \subset \FSL_{(n)} \subset \cdots \subset
\FSL_{(1)} \subset \FSL_{(0.5)} \subset \FSL_{(0)} \subset \CSL
\]
of the string link concordance group
\cite{Cochran-Orr-Teichner:1999-1,Harvey:2006-1}.  Here $\FSL_{(h)}$
($h\in \frac12 \Z_{\ge 0}$) denotes the subgroup (of concordance
classes) of string links whose closures are $(h)$-solvable in the
sense of~\cite{Cochran-Orr-Teichner:1999-1}.  Harvey proved an
interesting result on this filtration that $\FSL_{(n)}/\FSL_{(n.5)}$
is highly nontrivial~\cite{Harvey:2006-1}.  More precisely, she
considered the subgroup $\B(m)$ of (concordance classes of) boundary
string links in $\CSL$, and the induced filtration
$\{\BFSL_{(h)}=\FSL_{(h)} \cap \B(m)\}$ of $\B(m)$.  She defined a
real-valued group homomorphism
\[
\rho_n\colon \frac{\BFSL_{(n)}}{\BFSL_{(n.5)}\cdot\langle\text{local
    knots}\rangle} \to \R,
\]
and using it, proved that the abelianization of
$\smash{\BFSL_{(n)}/(\BFSL_{(n.5)}}\cdot\langle\text{local
  knots}\rangle)$ has infinite rank for any $m>1$ and~$n$.  (In the
original statement in~\cite{Harvey:2006-1}, $\smash{\BFSL_{(n+1)}}$
was used in place of $\smash{\BFSL_{(n.5)}}$; it was generalized to
the above form in a subsequent work of Cochran and
Harvey~\cite{Cochran-Harvey:2006-01}.)

Our invariant reveals further information on the filtration.  Let
$\hFSL_{(h)}=\FSL_{(h)} \cap \hCSL$.  We remark that $\BFSL_{(h)}$ is
a subgroup of $\hFSL_{(h)}$ since a boundary string link is an
$\hF$-string link.  It is known that there are $\hF$-(string) links
which are not concordant to any boundary (string)
link~\cite{Cochran-Orr:1993-1}; the examples
in~\cite{Cochran-Orr:1993-1} illustrate that $\BFSL_{(1)}$ is a proper
subgroup of $\hFSL_{(1)}$ for $m>1$.

The following result may be viewed as a refinement of
Theorem~\ref{theorem:lambda_T-homomorophism}:

\begin{theorem}
  \label{theorem:COT-filtration-and-lambda_T}
  For any $p$-structure $\T=(\{X_i\},\theta)$ of height $n$,
  $\lambda_\T$ gives rise to a homomorphism
  \[
  \hFSL_{(n)}/\hFSL_{(n.5)} \subset \hCSL/\hFSL_{(n.5)} \to L^0(\Q(\zeta_d)).
  \]
  In addition, if $\T$ is locally trivial, then $\lambda_\T$ induces a
  homomorphism
  \[
  \frac{\hFSL_{(n)}}{\hFSL_{(n.5)}\cdot\langle\text{local
      knots}\rangle} \subset
  \frac{\hCSL}{\hFSL_{(n.5)}\cdot\langle\text{local knots}\rangle} \to
  L^0(\Q(\zeta_d)).
  \]
\end{theorem}

Using this theorem we show that for any $m>1$ and $n$, the
abelianization of $\hFSL_{(n)}/(\hFSL_{(n.5)}\cdot\langle\text{local
  knots}\rangle)$ is of infinite rank (see
Theorem~\ref{theorem:Z-independence-of-string-links}).  As an
immediate corollary of (the proof of) this, we obtain an alternative
proof of Harvey's result that the abelianization of
$\BFSL_{(n)}/(\BFSL_{(n.5)}\cdot\langle\text{local knots}\rangle)$ has
infinite rank.  Moreover, we prove that the kernel of Harvey's
homomorphism $\rho_n$ is large:

\begin{theorem}
  \label{theorem:kernel-of-harvey-invariant}
  For any $m>1$ and $n$, the abelianization of the kernel of $\rho_n$
  on $\BFSL_{(n)}/(\BFSL_{(n.5)}\cdot\langle\text{local
    knots}\rangle)$ has infinite rank.
\end{theorem}

In the proof of Theorem~\ref{theorem:kernel-of-harvey-invariant} we
construct concrete examples of (boundary) string links which have
vanishing $\rho_n$-invariants but are linearly independent (over $\Z$)
in the abelianization.  We remark that recently Cochran, Harvey, and
Leidy \cite{Cochran-Harvey-Leidy:2007-1} have announced a (different)
proof of the following statement, which is an immediate consequence of
Theorem~\ref{theorem:kernel-of-harvey-invariant}: viewing $\rho_n$ as
a homomorphism on $\BFSL_{(n)}/\BFSL_{(n.5)}$, its kernel is
infinitely generated (without taking the abelianization and without
taking the quotient by local knots).

% We remark that, as a part of the proof of
% Theorem~\ref{theorem:COT-filtration-and-lambda_T}, we prove the
% following result which is interesting on its own.  As a $p$-analogue
% of the notion of $(h)$-solvable 3-manifolds
% in~\cite{Cochran-Orr-Teichner:1999-1}, we think of
% $\Z_{(p)}$-cofficient $(h)$-solvable 3-manifolds.  (For a precise
% description, refer to
% Definition~\ref{definition:Zp-coefficient-solution}.)

% \begin{theorem}[Covering Solution Theorem]
%   \label{theorem:covering-solution}
%   Suppose $M$ is $\Z_{(p)}$-coefficient $(h)$-solvable with $h\ge 1$,
%   $\phi\colon \pi_1(M)\to \Gamma$ is a homomorphism onto an abelian
%   $p$-group~$\Gamma$, and both $M$ and $M_\Gamma$ have $p$-torsion
%   free $H_1(-;\Z)$, where $M_\Gamma$ is the $\Gamma$-cover of $M$
%   determined by~$\phi$.  Then $M_\Gamma$ is $\Z_{(p)}$-coefficient
%   $(h-1)$-solvable.
% \end{theorem}

As a part of the proof
Theorem~\ref{theorem:COT-filtration-and-lambda_T}, we investigate the
relationship between the solvability of a 3-manifold and its
(iterated) covers.  A related result we call \emph{Covering Solution
  Theorem} seems interesting by its own.  Roughly speaking, it says:
\emph{taking an abelian $p$-cover, the solvability decreases by one.}
(We need some technical assumptions; for precise statements, refer to
Definition~\ref{definition:Zp-coefficient-solution} and
Theorem~\ref{theorem:covering-solution}.)  Further applications of
Covering Solution Theorem will be discussed in later papers.

\subsection{Disk basings and independence of links}

We also investigate the structure of the
Cochran-Orr-Teichner filtration of (spherical) links.  Let $\CL(m)$ be
the set of concordance classes of $m$-component links in~$S^3$, and
let
\[
\cdots \subset \FL_{(n.5)} \subset \FL_{(n)} \subset \cdots \subset
\FL_{(1)} \subset \FL_{(0.5)} \subset \FL_{(0)} \subset \CL(m)
\]
be the Cochran-Orr-Teichner filtration of $\CL(m)$, namely,
$\FL_{(h)}$ is the collection (of concordance classes) of
$(h)$-solvable links in the sense of
\cite{Cochran-Orr-Teichner:1999-1}.

In contrast to $\CSL(m)$, $\CL(m)$ does not have a natural group
structure for $m>1$.  One can think of a connected sum of two links,
but in general it does not give a unique (concordance class of a)
link.  To define a connected sum one may pass through string links; by
choosing a ``disk basing'' of a link $L$ in the sense of
\cite{Habegger-Lin:1990-1}, one obtains a string link $\beta$ whose
closure is~$L$.  Here a disk basing is an embedded 2-disk in $S^3$
which meets components of $L$ transversely at prescribed positive
intersection points, and $\beta$ is obtained by cutting $S^3$ along
the 2-disk.  Given two links $L_1$ and $L_2$, a connected sum of $L_1$
and $L_2$ is defined to be the closure of the product of two string
links obtained from $L_1$ and $L_2$ by choosing disk basings.
Similarly we define a connected sum of finitely many links by choosing
disk basings and an order of the links.

Regarding the role of disk basings, we think of a strong notion of
``independence'' of links as below.  For a link $L$ and $a\in\Z$, we
denote by $aL$ $a$ copies of $L$ if $a\ge 0$, and $|a|$ copies of $-L$
for $a<0$, where $-L$ designates the inverse of~$L$.  A family
$\{L_i\}_{i\in I}$ of links indexed by a set $I$ is called
\emph{independent over $\Z$ modulo $\FL_{(h)}$ and local knots} if the
following holds: for any sequence $\{a_i\}_{i\in I}$ of integers which
are all zero but finitely many, if a connected sum of the $a_i L_i$
($i\in I$) with some local knots added is $(h)$-solvable, then $a_i=0$
for all $i\in I$.  Roughly, this means that any connected sum of
copies of a link in $\{L_i\}_{i\in I}$ is never obtained as a
connected sum of copies of other links in $\{L_i\}_{i\in I}$ even
modulo $(h)$-solvability and local knots.

\begin{theorem}
  \label{theorem:independent-links-over-Z}
  For any $m>1$ and $n$, there are infinitely many $m$-component links
  which are in $\FL_{(n)}$, have vanishing $\rho_n$-invariant, and are
  independent over $\Z$ modulo $\FL_{(n.5)}$ and local knots.
\end{theorem}

In \cite{Cha:2007-1}, it was shown that the $n$th iterated Bing
doubles of some knots are in $\FL_{(n)}$ and have vanishing
$\rho_n$-invariant but are not in $\FL_{(n+1)}$.  Using the techniques
used in the proof Theorem~\ref{theorem:independent-links-over-Z}, we
generalize this result by showing that the knots can be chosen in such
a way that their $n$th iterated Bing doubles satisfy the conclusion of
Theorem~\ref{theorem:independent-links-over-Z}.

In \cite{Cha:2007-1}, it was also shown that there are 2-torsion links
in an arbitrary depth of the Cochran-Orr-Teichner filtration; the
examples are the $n$th iterated Bing doubles of certain (negatively)
amphichiral knots.  (We call a link $L$ \emph{2-torsion} if a
connected sum of $L$ and $L$ itself is a slice link; we remark that
the $\rho_n$-invariants vanish for any 2-torsion link.)  We consider
the following analogous notion of independence for 2-torsion links: a
family $\{L_i\}_{i\in I}$ of 2-torsion links is said to be
\emph{independent over $\Z_2$ modulo $\FL_{(h)}$ and local knots} if a
subset $J$ of $I$ is an empty set whenever a connected sum of
$\{L_j\}_{j\in J}$ with some local knots added is $(h)$-solvable.  We
show that the 2-torsion examples of iterated Bing doubles in
\cite{Cha:2007-1} can be chosen in such a way that they are
independent in this sense:

\begin{theorem}
  \label{theorem:independent-links-over-Z2}
  For any $n$, there are infinitely many amphichiral knots whose $n$th
  iterated Bing doubles are 2-torsion, in $\FL_{(n)}$, and independent
  over $\Z_2$ modulo $\FL_{(n+1.5)}$ and local knots.
\end{theorem}

The proofs of Theorems~\ref{theorem:independent-links-over-Z}
and~\ref{theorem:independent-links-over-Z2} involve a careful analysis
of the effect of a disk basing change on $p$-structures, utilizing
results of Habegger and
Lin~\cite{Habegger-Lin:1990-1,Habegger-Lin:1998-1}.  See
Section~\ref{section:independence-of-links} for details.

\section{Invariants of string links}
\label{section:invariant-of-string-links}

We review basic definitions and fix notations for string links.  Fix
$m$ distinct points $p_1,\ldots,p_m$ in the interior of~$D^2$.  An
\emph{$m$-component string link} (or \emph{$m$-string link}) is
defined to be the image of a locally flat proper embedding of
$\{1,\ldots,m\}\times [0,1]$ into $D^2\times[0,1]$ sending $(i,t)$ to
$(p_i,t)$ for $i=1,\ldots,m$ and $t=0,1$.  The product of two string
links is defined by juxtaposition.  The concordance classes of string
links (in the sense of \cite{Habegger-Lin:1998-1}) form a group under
the product operation; the inverse $-\beta$ is the mirror image of
$\beta$ (about $D^2\times\{\frac{1}{2}\}$) with reversed orientation.
We denote this group by $\CSL=\CSL(m)$ as in the introduction.

Gluing $D^2\times\{0\}$ and $D^2\times\{1\}$ along the identity map
of $D^2$ and then filling it in with a solid torus in such a way that
the image of $\{*\}\times[0,1]$ bounds a disk for $*\in \partial D^2$,
we obtain from $\beta$ a link $L$ in $S^3$, which is called the
\emph{closure} of~$\beta$.  It is known that the closure $L$ is a
slice link if and only if $\beta$ is concordant to a trivial string
link~\cite{Habegger-Lin:1998-1}.  We denote the exterior of $\beta$
and $L$ by $E_\beta$ and $E_L$, respectively.  Also, the surgery
manifold of $L$ is denoted by $M_\beta$ or $M_L$.  There are natural
inclusions $E_\beta \to E_L \to M_\beta$.

As in the introduction, let $X=\bigvee^m S^1$, the wedge of $m$
circles.  Fix an embedding $f\colon X\to D^2-\{p_1,\ldots,p_m\}$ such
that the wedge point $*\in X$ is sent to a fixed basepoint on
$\partial D^2$ and $x \to (f(x),0)$ defines a map $\mu\colon X \to
E_\beta$ which sends the $i$th circle of $X$ to a positive meridian of
the $i$th component of~$\beta$.  We call $\mu$ the \emph{preferred
  meridian map} of~$\beta$.  Sometimes the composition
\[
X \xrightarrow{\mu} E_\beta \to M_{\beta}
\]
is also referred to as the preferred meridian map.

We call a string link $\beta$ an \emph{$\hF$-string link} if its
closure $L$ is an $\hF$-link in the sense of
Levine~\cite{Levine:1989-1}, that is, the composition
\[
X \xrightarrow{\mu} E_\beta \to E_{L}
\]
induces an isomorphism $\widehat{\pi_1(X)} \to \widehat{\pi_1(E_L)}$
and the preferred longitudes of $L$ are in the kernel of the natural
map $\pi_1(E_L) \to \widehat{\pi_1(E_L)}$.  Here $\widehat G$
denotes the algebraic closure of a group $G$ defined by
Levine~\cite{Levine:1989-1}.

As a (potentially) generalized notion, we say that $\beta$ is a
\emph{$\Z_{(p)}$-coefficient $\hF$-string link} if $\beta$ satisfies
the defining condition of an $\hF$-string link with Levine's algebraic
closure replaced by the ``algebraic closure with respect to
$\Z_{(p)}$-coefficients'' defined in~\cite{Cha:2004-1}.  Equivalently,
$\beta$ is a \emph{$\Z_{(p)}$-coefficient $\hF$-string link} if its
closure is a \emph{$\Z_{(p)}$-coefficient $\hF$-link} in the sense
of~\cite{Cha:2007-1}.  Since an $\hF$-string link is a
$\Z_{(p)}$-coefficient $\hF$-string link for any $p$ and since all
results in this paper hold for the latter as well as the former, from
now on $\widehat G$ denotes the algebraic closure of $G$ with respect
to $\Z_{(p)}$-coefficients, and an ``$\hF$-string link'' designates a
$\Z_{(p)}$-coefficient $\hF$-string link, as an abuse of terminology.
It can be seen that the set $\hCSL$ of classes of $\hF$-string links
is closed under the group operations of $\CSL$, that is, $\hCSL$ is a
subgroup in~$\CSL$.

Suppose $Y \to Z$ is a map between CW-complexes.  Then a $\Z_d$-valued
$p$-structure of height $n$ for $Z$ induces a $\Z_d$-valued
$p$-structure of height $n$ for $Y$ via pullback along $Y\to Z$.  We
recall from \cite{Cha:2007-1} that $Y\to Z$ is called a
\emph{$p$-tower map} if pullback gives rise to a 1-1 correspondence
between $\Z_d$-valued $p$-structures of height $n$ for $Y$ and $Z$ for
any $n$ and~$d$.  For a more precise description, see Definition 3.4
of~\cite{Cha:2007-1}.  The following was proved in~\cite{Cha:2007-1}:

\begin{lemma}[Proposition 6.3 of \cite{Cha:2007-1}]
  \label{lemma:p-tower-map-of-surgery-mfd}
  For any $\hF$-string link $\beta$, any meridian map $X \to M_\beta$
  is a $p$-tower map.  In particular, the preferred meridian map into
  $M_\beta$ is a $p$-tower map.
\end{lemma}

% For later use, we remark that in \cite{Cha:2007-1}
% Lemma~\ref{lemma:p-tower-map-of-surgery-mfd} is proved by showing two
% facts: (i) $X \to M_\beta$ induces an isomorphism $\widehat{\pi_1(X)}
% \to \widehat{\pi_1(M_\beta)}$, and (ii) for CW-complexes $Y$ and $Z$
% with finite 2-skeletons, if a map $f\colon Y \to Z$ induces an
% isomorphism on $\widehat{\pi_1(-)}$, then $f$ is a $p$-tower map.

Suppose $\T=(\{X_k\}, \theta)$ is a $\Z_d$-valued $p$-structure of
height $n$ for~$X$.  Since the preferred meridian map $X \to M_\beta$
of $\beta$ is a $p$-tower map, there is a uniquely determined
$\Z_d$-valued $p$-structure $(\{M_k\}, \phi)$ of height $n$ for
$M_\beta$ which induces $\T$ via pullback.  In \cite[Definition
2.2]{Cha:2007-1}, an invariant $\lambda(M_n,\phi)\in L^0(\Q(\zeta_d))$
is defined when $(M_n,\phi)=0$ in the topological bordism group
$\Omega_3^{top}(B\Z_d)$.

\begin{lemma}
  \label{lemma:well-definedness-of-lambda_T}
  If $\beta$ is an $\hF$-string link, then $(M_n,\phi)=0$ in
  $\Omega_3^{top}(B\Z_d)$.
\end{lemma}

The proof of Lemma~\ref{lemma:well-definedness-of-lambda_T} is
postponed.  By Lemma~\ref{lemma:well-definedness-of-lambda_T}, the
following definition is meaningful for any $\hF$-string link:

\begin{definition}
  For a $p$-structure $\T=(\{X_k\}, \theta)$ for $X$ and an
  $\hF$-string link $\beta$, we define
  \[
  \lambda_\T(\beta)=\lambda(M_n,\phi) \in L^0(\Q(\zeta_d))
  \]
  where $(\{M_k\},\phi)$ is the $p$-structure for $M_\beta$ induced by
  $\T$ as above.
\end{definition}

The main aim of this section is to prove the following additivity:
\begin{theorem}
  \label{theorem:additivity-of-lambda_T}
  If $\beta$ and $\beta'$ are $\hF$-string links, then for any
  $p$-structure $\T$, we have
  \[
  \lambda_\T(\beta\cdot \beta') = \lambda_\T(\beta) +
  \lambda_\T(\beta') \quad\text{in }L^0(\Q(\zeta_d)).
  \]
\end{theorem}

By \cite[Theorem~6.2]{Cha:2007-1}, $\lambda_T(\beta)=0$ if $\beta$ is
concordant to a trivial string link.  Combining this with
Theorem~\ref{theorem:additivity-of-lambda_T}, we obtain (the first
part of) Theorem~\ref{theorem:lambda_T-homomorophism} stated in the
introduction: for any $p$-structure $\T=(\{X_k\}, \theta)$, the
invariant $\lambda_\T(-)$ gives rise to a group homomorphism
\[
\lambda_\T\colon \hCSL \to L^0(\Q(\zeta_d)).
\]

The remaining part of this section is devoted to the proof of
Lemma~\ref{lemma:well-definedness-of-lambda_T} and
Theorem~\ref{theorem:additivity-of-lambda_T}.

\begin{proof}[Proof of Lemma~\ref{lemma:well-definedness-of-lambda_T}]
  Suppose $\beta$ is an $\hF$-string link.  Suppose $\{M_i\}$ is a
  $p$-tower of height $n$ for $M_\beta$ and $\phi\colon\pi_1(M_n) \to
  \Z_d$ is a character.  Let $\{X_i\}$ be the $p$-tower for $X$ and
  $\theta\colon \pi_1(X_n) \to \Z_d$ be the character which are
  induced via pullback along $X \to M_\beta$.  Since $X \to M_\beta$
  is a $p$-tower map, the lift $X_n \to M_n$ of $X\to M_\beta$ induces
  an isomorphism
  \[
  \Hom(\pi_1(M_n),\Z_{p^r}) \cong \Hom(\pi_1(X_n),\Z_{p^r})
  \]
  for any~$r$.  Since
  \[
  H_1(-)\otimes \Z_{p^r} \cong \Hom(H_1(-),\Z_{p^r})\cong
  \Hom(\pi_1(-),\Z_{p^r})
  \]
  we have
  \[
  H_1(M_n)\otimes \Z_{p^r} \cong H_1(X_n)\otimes \Z_{p^r}.
  \]
  Since $X_n$ is a 1-complex, $H_1(X_n)\otimes \Z_{p^r} =
  (\Z_{p^r})^s$ for all $r$, where $s=\beta_1(X_n)$ is the first Betti
  number.  Looking at the case $r=1$ and the case of a sufficiently
  large $r$, one can see that $H_1(M_n)$ is $p$-torsion free, that is,
  $H_1(M_n)$ has no nontrivial $p$-primary summand.  From this it
  follows that $\phi\colon \pi_1(M_n) \to \Z_d$ factors through
  $H_1(M_n)/\text{torsion}$, since $d$ is a power of~$p$.  Furthermore
  the induced map $H_1(M_n)/\text{torsion} \to \Z_d$ factors through
  the projection $\Z \to \Z_d$ since $H_1(M_n)/\text{torsion}$ is a
  free abelian group.

  From the above observation, it follows that $(M_n,\phi)\in
  \Omega^{top}_3(B\Z_d)$ is contained in the image of
  $\Omega^{top}_3(B\Z) \to \Omega^{top}_3(B\Z_d)$.  Therefore,
  $(M_n,\phi)=0$ in $\Omega^{top}_3(B\Z_d)$ since
  $\Omega^{top}_3(B\Z)=0$ (e.g., by the Atiyah-Hirzebruch spectral
  sequence).
\end{proof}

Let $M=M_\beta$, $M'=M_{\beta'}$, and $N=M_{\beta\cdot\beta'}$ be the
surgery manifolds of the closures of $\beta$, $\beta'$, and
$\beta\cdot\beta'$, respectively.  In order to prove
Theorem~\ref{theorem:additivity-of-lambda_T}, as in
\cite{Cochran-Orr-Teichner:2002-1,Harvey:2006-1} we consider a
``standard'' cobordism $W$ between $M\cup M'$ and $N$, which can be
described as follows: attaching a 1-handle to $(M\cup M')\times[0,1]$,
we obtain a cobordism from $M\cup M'$ to $M \# M'$.  And then
attaching $m$ 2-handles along the attaching spheres in $M\#M'$ as
illustrated in Figure~\ref{figure:surgery-manifold-cobordism}, we
obtain a cobordism from $M \# M'$ to~$N$.  (The last two surgery
diagrams in Figure~\ref{figure:surgery-manifold-cobordism} are
equivalent by handle sliding and cancellation.)  Our $W$ is obtained
by concatenating these cobordisms.

\begin{figure}[ht]
  \begin{center}
    \includegraphics[scale=.9]{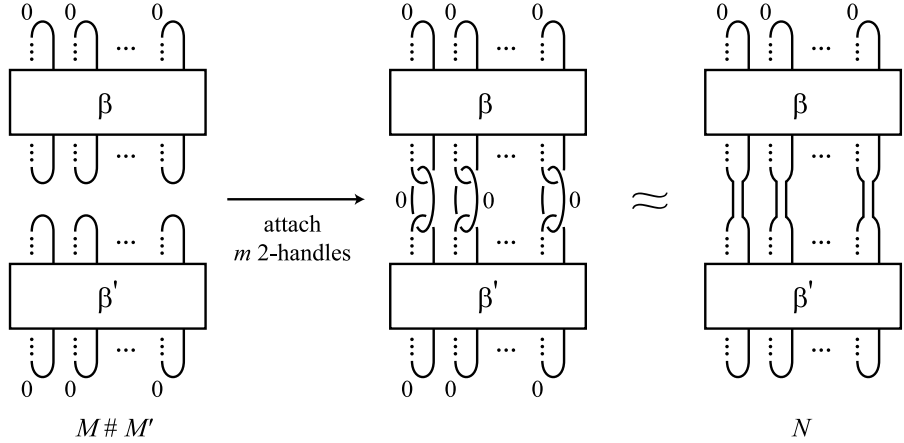}
  \end{center}
  \caption{A cobordism from $M \# M'$ to $N$}
  \label{figure:surgery-manifold-cobordism}
\end{figure}

Let $(\{M_k\},\phi)$, $(\{M_k'\},\phi')$ and $(\{N_k\},\psi)$ be the
$p$-structures of height $n$ for $M$, $M'$ and $N$, respectively,
which are determined by a given $p$-structure $\T=(\{X_k\}, \theta)$
of height $n$ for~$X$.  The first step of our proof is a construction
of a bordism between $(M_n,\phi)$, $(M_n',\phi')$, $(N_n,\psi)$
using~$W$.

\begin{lemma}
  \label{lemma:p-tower-bordism}
  There is a $p$-tower $\{W_k\}$ of $W$ which induces $\{M_k\}$,
  $\{M_k'\}$, and $\{N_k\}$ by pullback along the inclusions.
  Furthermore, there is a character $\pi_1(W_n) \to \Z_d$ which
  restricts to $\phi$, $\phi'$, and $\psi$ on $\pi_1(M_n)$,
  $\pi_1(M_n')$, and $\pi_1(N_n)$, respectively.
\end{lemma}

In order to prove Lemma~\ref{lemma:p-tower-bordism}, we need the
following facts which were proved
in~\cite{Levine:1989-1,Cha:2004-1,Cha:2007-1}: for CW-complexes or
groups $A$ and $B$, we say that $f\colon A \to B$ is \emph{2-connected
  with respect to $\Z_{(p)}$-coefficients} if $f$ induces an
isomorphism on $H_1(-;\Z_{(p)})$ and an epimorphism on
$H_2(-;\Z_{(p)})$.  We remark that the defining condition holds for
$\Z_{(p)}$-coefficients if and only if it holds for
$\Z_{p}$-coefficients.  In this section, for convenience we simply say
that a map is 2-connected when it is 2-connected with respect to
$\Z_{(p)}$ (or $\Z_p$)-coefficients.

\begin{lemma}
  \label{lemma:algebraic-closure-and-p-tower-map}
  In what follows $A$, $B$, and $B_i$ are CW-complexes with finite
  2-skeletons.
  \begin{enumerate}
  \item[(0)] For any group $G$, the natural map $G \to \widehat{G}$ is
    2-connected \cite{Levine:1989-1,Cha:2004-1}.
  \item If $f\colon A \to B$ is 2-connected, then $f$ induces an
    isomorphism $\widehat{\pi_1(A)} \to \widehat{\pi_1(B)}$
    \cite{Levine:1989-1,Cha:2004-1}.
  \item As a weak converse to (1), if $\{f_i\colon A \to B_i\}$ is a
    finite collection of maps inducing isomorphisms
    $\widehat{\pi_1(A)} \to \widehat{\pi_1(B_i)}$, then there is a
    2-connected map $g\colon A\to Z=$ (a $K(G,1)$-space with finite
    2-skeleton) such that for each $i$ there is a 2-connected map $B_i
    \to Z$ making the diagram
    \[
    \begin{diagram}
      \node{A} \arrow[2]{e,t}{f_i} \arrow{se,b}{g} \node[2]{B_i}
      \arrow{sw,..}
      \\
      \node[2]{Z}
    \end{diagram}
    \]
    commute \cite[Proof of Proposition~3.9]{Cha:2007-1}.
  \item Any $f\colon A \to B$ inducing an isomorphism
    $\widehat{\pi_1(A)} \to \widehat{\pi_1(B)}$ is a $p$-tower map
    \cite[Proposition~3.9]{Cha:2007-1}.
  \end{enumerate}
\end{lemma}

We remark that although \cite[Proof of Proposition~3.9]{Cha:2007-1}
discusses a special case of Lemma
\ref{lemma:algebraic-closure-and-p-tower-map} (2) that the family
$\{f_i\}$ consists of only one map, exactly the same argument proves
the above generalized case.

\begin{proof}[Proof of Lemma~\ref{lemma:p-tower-bordism}]
  Let $X'=X$ and $\mu \colon X\to E_\beta \to M$ and $\mu'\colon X'
  \to E_{\beta'} \to M'$ be the preferred meridian maps of $\beta$
  and~$\beta'$. By Lemma~\ref{lemma:p-tower-map-of-surgery-mfd}, $\mu$
  and $\mu'$ induce isomorphisms $\widehat{\pi_1(X)} \to
  \widehat{\pi_1(M)}$ and $\widehat{\pi_1(X')} \to
  \widehat{\pi_1(M')}$.  By
  Lemma~\ref{lemma:algebraic-closure-and-p-tower-map}~(2), it follows
  that there are 2-connected maps of $M$, $M'$ into a CW-complex $Z$
  with finite 2-skeleton making the following diagram commute:
  \[
  \begin{diagram} \dgARROWLENGTH=1.6em \dgHORIZPAD=5pt
    \node{X} \arrow{e,t}{\mu} \arrow[2]{s,=}
    \node{M} \arrow{s}\arrow{see}
    \\
    \node[2]{M\vee M'}\arrow{e} \node{W} \arrow{e,..}\node{Z}
    \\
    \node{X'} \arrow{e,t}{\mu'}
    \node{M'} \arrow{n}\arrow{nee}
  \end{diagram}
  \]
  Note that $W$ has the homotopy type of
  \[
  \big( M \vee M' \big) \cup \big(\text{$m$ 2-disks attached along
    $\mu_*(x_i^{\vphantom{\prime}})\mu'_*(x_i')^{-1}$} \big)
  \]
  where $x_i$ and $x_i'$ denote the paths representing the $i$th
  circle in $X$ and $X'$, respectively.  It follows that there is a
  map $W \to Z$ making the above diagram commute.

  Let $\mu''\colon X'' \to N$ be the preferred meridian map of
  $\beta\cdot\beta'$, where $X''=X$.  It can be seen from
  Figure~\ref{figure:surgery-manifold-cobordism} that the preferred
  meridians of $\beta$, $\beta'$, and $\beta\cdot\beta'$ are homotopic
  in~$W$.  Therefore, we have the following commutative diagram which
  extends the above one:
  \[
  \begin{diagram} \dgARROWLENGTH=1.6em \dgHORIZPAD=5pt
    \node{X} \arrow{e,t}{\mu} \arrow[2]{s,=}
    \node{M} \arrow{s}\arrow{see}
    \\
    \node[2]{M\vee M'}\arrow{e} \node{W}\arrow{e} \node{Z}
    \\
    \node{X'} \arrow{e,t}{\mu'} \arrow{s,=}
    \node{M'} \arrow{n}
    \\
    \node{X''} \arrow{e,t}{\mu''}
    \node{N} \arrow{nne}
  \end{diagram}
  \]
  
  By Lemma~\ref{lemma:algebraic-closure-and-p-tower-map}~(3), the
  composition map $X \to Z$ is a $p$-tower map.  So, given
  $\T=(\{X_k\},\theta)$, there is a $p$-tower $\{Z_k\}$ for $Z$ and a
  character of $\pi_1(Z_n)$ inducing $\{X_k\}$ and $\theta$ via
  pullback.  Now, by their definition, $\{M_k\}$, $\{M'_k\}$, and
  $\{N_k\}$ are identical with the pullback $p$-towers induced by
  $\{Z_k\}$, and similarly for characters.  It follows that the
  $p$-tower $\{W_k\}$ for $W$ induced by $\{Z_k\}$ and the character
  of $\pi_1(W_n)$ induced by that of $\pi_1(Z_n)$ have the claimed
  properties.
\end{proof}

Let $\{(M\cup M')_k\}$ and $\{(M\vee M')_k\}$ be the pullback
$p$-towers of $M\cup M'$ and $M\vee M'$ determined by $\{Z_k\}$ in the
above proof.

\begin{lemma}
  \label{lemma:triviality-of-intersection-form}
  $H_2((M\cup M')_k;\Z_{(p)}) \to H_2(W_k;\Z_{(p)})$ is surjective for
  all~$k$.
\end{lemma}

We remark that our argument below also shows the conclusion for
$\Z_p$-coefficients.

\begin{proof}
  Since $(M\vee M')_k$ has the homotopy type of $(M\cup M')_k \cup$
  (1-cells),
  \[
  H_2((M\cup M')_k;\Z_{(p)}) \to H_2((M\vee M')_k;\Z_{(p)})
  \]
  is surjective by the long exact sequence of the pair $((M\vee M')_k,
  (M\cup M')_k)$.  Therefore, it suffices to show that
  \[
  H_2((M\vee M')_k;\Z_{(p)}) \to H_2(W_k;\Z_{(p)})
  \]
  is surjective.  Let $Y=\bigvee^m S^1$, and let $Y \to X\vee X'$ be
  the map sending the $i$th circle $y_i$ of $Y$ to
  $x_i^{\vphantom{\prime}}(x_i')^{-1}$.  (Recall that $x_i$ and $x_i'$
  represent the $i$th circles of $X$ and $X'$, respectively.)  Since
  $W$ has the homotopy type of
  \[
  (M\vee M') \cup \bigg(\smash[b]{
  \begin{tabular}{cc}
    $m$ 2-disks attached along the image of $y_i$ under \\
    $Y \to X\vee X' \smash{\xrightarrow[\mu\vee \mu']{}} M\vee M'$
  \end{tabular}
  }\bigg)
  \]
  we can see that $W_k$ has the homotopy type of
  \[
  (M\vee M')_k \cup \bigg(\smash[b]{
  \begin{tabular}{cc}
    2-disks attached along the image of lifts of $y_i$ under \\
    $Y_k \to (X\vee X')_k \to (M\vee M')_k$
  \end{tabular}
  }\bigg)
  \]  
  where $\{(X\vee X')_k\}$ and $\{Y_k\}$ are the pullback $p$-towers
  of $X\vee X'$ and $Y$ determined by $\{Z_k\}$, respectively.  (In
  fact, $Y_k$ is a disjoint union of copies of $Y$ since $Y \to Z$
  factors through $W$ and $\pi_1(Y) \to \pi_1(W)$ is trivial.)  From
  this we obtain a Mayer-Vietoris exact sequence
  \[
  H_2((M\vee M')_k;\Z_{(p)}) \to H_2(W_k;\Z_{(p)}) \to
  H_1(Y_k;\Z_{(p)}) \to H_1((M\vee M')_k;\Z_{(p)}).
  \]
  We will complete the proof by showing that the rightmost map on
  $H_1(-;\Z_{(p)})$ is injective.  For this purpose, we factor $Y \to
  M\vee M'$ into a composition of 3 maps:
  \[
  Y \xrightarrow{\iota} Y\vee Y' \xrightarrow{\alpha} X \vee X'
  \xrightarrow{\mu\vee \mu'} M\vee M'
  \]
  where $Y'=Y$, the first map $\iota$ is an obvious inclusion into the
  first factor, and the second map $\alpha$ is defined by $y_i \mapsto
  x_i^{\vphantom{\prime}}(x_i')^{-1}$ and $y_i' \mapsto x_i'$ (here
  $y_i'$ represents the $i$th circle of~$Y'$).  We investigate the
  induced maps on pullback $p$-towers $\{Y_k\}$, $\{(Y\vee Y')_k\}$,
  $\{(X\vee X')_k\}$, and $\{(M\vee M')_k\}$ determined by $\{Z_k\}$:
  \begin{enumerate}
  \item $\iota$ induces an injection
    \[
    H_1(Y_k;\Z_{(p)}) \to H_1((Y\vee Y')_k;\Z_{(p)})
    \]
    since $Y_k \to (Y\vee Y')_k$ is an embedding between 1-complexes,
    being the lift of an embedding~$\iota$.
  \item $\alpha$ induces an isomorphism
    \[
    H_1((Y\vee Y')_k;\Z_{(p)}) \to H_1((X\vee X')_k;\Z_{(p)}).
    \]
    For, $\alpha$ induces an isomorphism on $\widehat{\pi_1(-)}$ by
    Lemma~\ref{lemma:algebraic-closure-and-p-tower-map}~(1), since
    $\alpha$ is an $H_1$-isomorphism and $H_2(X\vee X';\Z_{(p)})=0$.  By
    Lemma~\ref{lemma:algebraic-closure-of-p-cover} stated below, it
    follows that
    \[
    (Y \vee Y')_k \to (X\vee X')_k
    \]
    induces an isomorphism on $\widehat{\pi_1(-)}$.  From this the
    claim follows, by
    Lemma \ref{lemma:algebraic-closure-and-p-tower-map}~(0).
  \item $\mu\vee\mu'$ induces an isomorphism
    \[
    H_1((X\vee X')_k;\Z_{(p)}) \to H_1((M\vee M')_k;\Z_{(p)}).
    \]
    For, for the CW-complex $Z$ in the proof of the previous lemma, we
    have a commutative diagram
    \[
    \begin{diagram}
      \node{X\vee X'} \arrow[2]{e,t}{\mu\vee\mu'} \arrow{se}
      \node[2]{M\vee M'} \arrow{sw}
      \\
      \node[2]{Z\vee Z}
    \end{diagram}
    \]
    with maps into $Z\vee Z$ 2-connected.  Applying
    Lemma~\ref{lemma:algebraic-closure-and-p-tower-map}~(1), it
    follows that $\mu\vee\mu'$ induces an isomorphism on
    $\widehat{\pi_1(-)}$.  Now the claim is shown by
    Lemma~\ref{lemma:algebraic-closure-of-p-cover} and
    Lemma~\ref{lemma:algebraic-closure-and-p-tower-map}~(0) as we did
    above.
  \end{enumerate}    
  Combining (1), (2), and (3), it follows that
  \[
  H_1(Y_k;\Z_{(p)}) \to H_1((M\vee M')_k;\Z_{(p)})
  \]
  is injective.
\end{proof}

\begin{lemma}
  \label{lemma:algebraic-closure-of-p-cover}
  Suppose $A$ and $B$ are CW-complexes with finite 2-skeletons and $A
  \to B$ is a map inducing an isomorphism on $\widehat{\pi_1(-)}$.
  Suppose $\pi_1(B) \to \Gamma$ is a map into an abelian $p$-group
  $\Gamma$, and $A_\Gamma$ and $B_\Gamma$ are covers of $A$ and $B$,
  respectively, induced by $\pi_1(A) \to \pi_1(B) \to \Gamma$.  Then
  the lift $A_\Gamma \to B_\Gamma$ induces an isomorphism on
  $\widehat{\pi_1(-)}$.  Consequently, for any $p$-tower $\{B_k\}$ of
  $B$ and the pullback $p$-tower $\{A_k\}$ of $A$, $A_k \to B_k$
  induces an isomorphism on $\widehat{\pi_1(-)}$.
\end{lemma}

\begin{proof}
  By Lemma~\ref{lemma:algebraic-closure-and-p-tower-map} (2), there is
  a commutative diagram
  \[
  \begin{diagram}
    \node{A} \arrow[2]{e} \arrow{se} 
    \node[2]{B} \arrow{sw}
    \\
    \node[2]{Z}
  \end{diagram}
  \]
  where $Z$ has finite 2-skeleton and vertical maps are 2-connected.
  Since $H_1(B;\Z_{p^r}) \cong H_1(Z;\Z_{p^r})$ for all $r$, the given
  $\pi_1(B) \to \Gamma$ factors through~$\pi_1(Z)$.  So, taking
  $\Gamma$-covers $A_\Gamma$, $B_\Gamma$, and $Z_\Gamma$ of $A$, $B$,
  and $Z$, we obtain a commutative diagram
  \[
  \begin{diagram}
    \node{A_\Gamma} \arrow[2]{e} \arrow{se} 
    \node[2]{B_\Gamma} \arrow{sw}
    \\
    \node[2]{Z_\Gamma}
  \end{diagram}.
  \]
  Since $\Gamma$ is a $p$-group and the maps into $Z$ in the previous
  diagram are 2-connected, the maps into $Z_\Gamma$ in this diagram
  are 2-connected too, by Levine's result \cite[Proof of Proposition
  3.2]{Levine:1994-1}.  (Refer to \cite[Lemma 3.2 and
  3.3]{Cha:2007-1}, \cite[Corollary 4.13]{Cochran-Harvey:2007-01} for
  statements that apply to our case directly.  See also
  \cite[Corollary 4.13]{Cochran-Harvey:2007-01}.)  By
  Lemma~\ref{lemma:algebraic-closure-and-p-tower-map} (1), it follows
  that $\widehat{\pi_1(A_\Gamma)} \cong \widehat{\pi_1(Z_\Gamma)}
  \cong \widehat{\pi_1(B_\Gamma)}$.
\end{proof}

Now we are ready to prove
Theorem~\ref{theorem:additivity-of-lambda_T}.

\begin{proof}[Proof of Theorem~\ref{theorem:additivity-of-lambda_T}]
  Suppose two $\hF$-string links $\beta$, $\beta'$ and a $p$-structure
  $\T=(\{X_k\},\theta)$ for $X$ are given.  Recall our notation:
  $(\{M_k\},\phi)$, $(\{M_k'\},\phi')$, $(\{N_k\},\psi)$ are the
  $p$-structures of surgery manifolds of $\beta$, $\beta'$,
  $\beta\cdot\beta'$ which are determined by the given $p$-structure
  $\T$, respectively.

  We need to show that
  $\lambda(M_n,\phi)+\lambda(M_n',\phi')=\lambda(N_n,\psi)$.  By
  Lemma~\ref{lemma:p-tower-bordism}, there is a bordism $W_n$ endowed
  with a character $\pi_1(W_n) \to \Z_d$ between $(\M_n,\phi) \cup
  (M_n',\phi')$ and $(N_n,\psi)$.  By the definition of $\lambda(-,-)$
  (see Definition 2.2 in~\cite{Cha:2007-1}), we have
  \[
  \lambda(M_n,\phi)+\lambda(M_n',\phi') - \lambda(N_n,\psi) =
  [\lambda_{\Q(\zeta_d)}(W_n)]-[\lambda_\Q(W_n)]
  \]
  where $[\lambda_{\Q(\zeta_d)}(W_n)]$ is the Witt class of (the
  nonsingular part of) the $\Q(\zeta_d)$-valued intersection form on
  $H_2(W_n;\Q(\zeta_d))$, and $[\lambda_{\Q}(W_n)]$ is the Witt class
  of (the nonsingular part of) the ordinary intersection form on
  $H_2(W_n;\Q)$.

  By Lemma~\ref{lemma:triviality-of-intersection-form},
  \[
  H_2((M\cup M)_n;\Z_{(p)}) \to H_2(W_n;\Z_{(p)})
  \]
  is surjective, and so is for
  $\Q$-coefficients.  It follows that $[\lambda_\Q(W_n)]=0$.

  Let $(M\cup M')_{n+1}$ be the $\Z_d$-cover of $(M\cup M')_n$
  determined by $\phi$ and $\phi'$, and denote the $\Z_d$-cover of
  $W_n$ by $W_{n+1}$ similarly.  Note that
  $H_2(W_n;\Q[\Z_d])=H_2(W_{n+1};\Q)$ and similarly for $M\cup M'$.
  So, applying Lemma~\ref{lemma:triviality-of-intersection-form} for
  $k=n+1$,
  \[
  H_2((M\cup M)_n;\Q[\Z_d]) \to H_2(W_n;\Q[\Z_d])
  \]
  is surjective.  Since $\Q(\zeta_d)$ is $\Q[\Z_d]$-flat,
  \[
  H_*(-;\Q(\zeta_d)) = H_*(-;\Q[\Z_d])\otimes_{\Q[\Z_d]}\Q(\zeta_d)
  \]
  and thus the surjectivity on $H_2$ holds for
  $\Q(\zeta_d)$-coefficients as well as~$\Q[\Z_d]$.  It follows that
  $[\lambda_{\Q(\zeta_d)}(W_n)]=0$.
\end{proof}

\section{Obstructions to being $(n.5)$-solvable}
\label{section:obstruction-to-n.5-solv}

In this section we prove that the Hirzebruch type invariants of
$\hF$-links give obstructions to being $(n.5)$-solvable, sharpening
the results on $(n+1)$-solvability in~\cite{Cha:2007-1}.  Our result
is best described in terms of the following $p$-analogue of the
integral (or rational) $(h)$-solvability.  Denote the $n$th lower
central subgroup of a group $G$ by $G^{(n)}$, i.e., $G^{(0)}=G$ and
$G^{(n+1)} = [G^{(n)}, G^{(n)}]$.

\begin{definition}
  \label{definition:Zp-coefficient-solution}
  Suppose $M$ is a closed 3-manifold.  A 4-manifold $W$ bounded by $M$
  is called a \emph{$\Z_{(p)}$-coefficient $(n)$-solution of $M$} if
  the following holds:
  \begin{enumerate}
  \item[(1)] $W$ is a $\Z_{(p)}$-coefficient $H_1$-bordism, i.e.,
    $H_1(M;\Z_{(p)}) \to H_1(W;\Z_{(p)})$ is an isomorphism.
  \item[(2)] There exist $u_1,\ldots, u_r, v_1,\ldots,v_r \in
    H_2(W;\Z_{(p)}[\pi/\pi^{(n)}])$, where $\pi=\pi_1(W)$ and
    $r=\frac{1}{2}\beta_2(W)$, such that the
    $\Z_{(p)}[\pi/\pi^{(n)}]$-valued intersection form
    $\lambda_W^{(n)}$ on $H_2(W;\Z_{(p)}[\pi/\pi^{(n)}])$ satisfies
    $\lambda_W^{(n)}(u_i,u_j)=0$ and
    $\lambda_W^{(n)}(u_i,v_j)=\delta_{ij}$.
  \end{enumerate}
  If, in addition, the following holds, then $W$ is called a
  \emph{$\Z_{(p)}$-coefficient $(n.5)$-solution of~$M$}:
  \begin{enumerate}
  \item[(3)] There exist $\tilde u_1,\ldots, \tilde u_r \in
    H_2(W;\Z_{(p)}[\pi/\pi^{(n+1)}])$ such that
    $\lambda_W^{(n+1)}(\tilde u_i, \tilde u_j)=0$ and $u_i$ is the image of
    $\tilde u_i$.
  \end{enumerate}
  If there is a $\Z_{(p)}$-coefficient $(h)$-solution of $M$ ($h\in
  \frac{1}{2}\Z_{\ge 0}$), then $M$ is said to be
  \emph{$\Z_{(p)}$-coefficient $(h)$-solvable}.  A (string) link is
  called \emph{$\Z_{(p)}$-coefficient $(h)$-solvable} if the surgery
  manifold (of its closure) is $\Z_{(p)}$-coefficient $(h)$-solvable.
\end{definition}

Obviously an integral $(h)$-solution defined
in~\cite{Cochran-Orr-Teichner:1999-1} is a $\Z_{(p)}$-coefficient
$(h)$-solution.  In this section, as an abuse of terminology, an
$(h)$-solution always designates a $\Z_{(p)}$-coefficient
$(h)$-solution, and similarly for an $(h)$-solvable 3-manifold.

For a given $\phi\colon \pi_1(M) \to \Z_{p^a}$, denote
$\Gamma_j=\Z_{p_j}$ ($j=0,1,\ldots,a$) and let $M_{\Gamma_j}$ be the
cover of $M$ determined by
\[
\pi_1(M) \xrightarrow{\phi} \Z_{p^a} \xrightarrow{\text{proj.}}
\Z_{p^a}/p^j\Z_{p^a}=\Gamma_j.
\]

\begin{theorem}
  \label{theorem:n.5-solvaility-obstruction}
  Suppose $(\{M_i\},\phi)$ is a $p$-structure of height $n$ for~$M$.
  If $M$ is $(n.5)$-solvable, $H_1(M_i)$ is $p$-torsion free for all
  $i$, and $\beta_1((M_n)_{\Gamma_j})-1=|\Gamma_j| (\beta_1(M_n)-1)$
  for $j=0,1,\ldots,a$, then $\lambda(M_n,\phi)$ is well-defined and
  vanishes.
\end{theorem}

Before proving Theorem~\ref{theorem:n.5-solvaility-obstruction}, we
discuss its consequences for links:

\begin{corollary}
  \label{corollary:n.5-solvability-obstruction-for-links}
  Suppose $(\{M_i\},\phi)$ is a $p$-structure of height $n$ for the
  surgery manifold of an $\hF$-link~$L$.  If $L$ is $(n.5)$-solvable,
  then $\lambda(M_n,\phi)=0$.
\end{corollary}

\begin{corollary}
  \label{corollary:n.5-solvability-obstruction-for-string-links}
  If $\beta$ is an $(n.5)$-solvable $\hF$-string link, then for any
  $p$-structure $\T=(\{X_i\},\theta)$ of height $n$ for $X$,
  $\lambda_\T(\beta)=0$.
\end{corollary}

\begin{proof}[Proof of
  Corollaries~\ref{corollary:n.5-solvability-obstruction-for-links}
  and~\ref{corollary:n.5-solvability-obstruction-for-string-links}]
  First note that
  Corollary~\ref{corollary:n.5-solvability-obstruction-for-string-links}
  follows immediately from
  Corollary~\ref{corollary:n.5-solvability-obstruction-for-links}.  To
  prove
  Corollary~\ref{corollary:n.5-solvability-obstruction-for-links},
  recall that in the proof of
  Lemma~\ref{lemma:well-definedness-of-lambda_T} we have shown the
  following: if $\{M_i\}$ is a $p$-tower of the surgery manifold $M$
  of an $\hF$-link, then $H_1(M_i)$ is $p$-torsion free and a meridian
  map $X=\bigvee^m S^1 \to M$ induces a pullback $p$-tower $\{X_i\}$
  such that $H_1(M_i)\otimes \Z_p \cong H_1(X_i)\otimes \Z_p$.  From
  this it follows that $\beta_1(M_i)=\beta_1(X_i)$.  A straightforward
  Euler characteristic computation (e.g., see \cite[Corollary
  6.5]{Cha:2007-1}) for the cover $X_i$ of $X$ shows that
  $\beta_1(X_i)-1=d(\beta_1(X)-1)$ when $X_i$ is a $d$-fold cover
  of~$X$.  It follows that
  Theorem~\ref{theorem:n.5-solvaility-obstruction} applies to conclude
  that $\lambda(M_n,\phi)=0$.
\end{proof}

Recall that $\hCSL$ is the subgroup of the classes of $\hF$-string
links in $\CSL$, $\FSL_{(h)}$ is the subgroup of the classes of
integrally $(h)$-solvable string links in $\CSL$ in the sense of
\cite{Cochran-Orr-Teichner:1999-1,Harvey:2006-1}, and
$\hFSL_{(n)}=\hCSL \cap \FSL_{(n)}$.  Since an integral $(h)$-solution
is a $\Z_{(p)}$-coefficient $(h)$-solution, as an immediate
consequence of
Corollary~\ref{corollary:n.5-solvability-obstruction-for-string-links}
and Theorem~\ref{theorem:additivity-of-lambda_T} we obtain the first
part of Theorem~\ref{theorem:COT-filtration-and-lambda_T} in the
introduction: $\lambda_\T(-)$ induces a group homomorphism
\[
\lambda_\T\colon \hCSL_{(n)}/\hFSL_{(n.5)} \to L^0(\Q(\zeta_d))
\]
whenever $\T$ is of height~$n$.

The remaining part of this section is devoted to the proof of
Theorem~\ref{theorem:n.5-solvaility-obstruction}.  The proof of
Theorem~\ref{theorem:n.5-solvaility-obstruction} consists of two
steps; first we investigate the solvability of an abelian $p$-cover of
an $(h)$-solvable 3-manifold, in order to reduce the general case to
the special case of $n=0$, and then we complete the proof by showing
Theorem~\ref{theorem:n.5-solvaility-obstruction} for the case of
$n=0$.  In the following two subsections we deal with each step.

\subsection{Solvability of abelian $p$-covers}

The first step of the proof of
Theorem~\ref{theorem:n.5-solvaility-obstruction}, namely
Theorem~\ref{theorem:covering-solution} below, is also interesting on
its own.  For a space $X$ endowed with a homomorphism $\pi_1(X)\to
\Gamma$, we denote the $\Gamma$-cover of $X$ by~$X_\Gamma$.

\begin{theorem}[Covering Solution Theorem]
  \label{theorem:covering-solution}
  Suppose $W$ is an $(h)$-solution for $M$ with $h\ge 1$, $\phi\colon
  \pi_1(M)\to \Gamma$ is a homomorphism onto an abelian
  $p$-group~$\Gamma$, and both $H_1(M)$ and $H_1(M_\Gamma)$ are
  $p$-torsion free.  Then $\phi$ extends to $\pi_1(W)$, and $W_\Gamma$
  is an $(h-1)$-solution of~$M_\Gamma$.
\end{theorem}

In the proof of Theorem~\ref{theorem:covering-solution}, we use the
following notations and observations.  For a fixed subring $R$ of $\Q$
and a 4-manifold $W$ endowed with $\pi_1(W) \to G$, we denote the
$RG$-valued intersection form on $H_2(W;RG)$ by~$\lambda_W^G$.  Here
$RG$ is regarded as a ring with involution $\bar g=g^{-1}$ as usual.
We adopt the convention that $H_*(W;RG)$ is a right $RG$-module and
$\lambda_W^G$ is defined by
\[
\lambda_W^G(x,y)=\sum_{g\in G} (x\cdot yg)g
\]
where $\cdot$ designates the ordinary intersection number.  Then
$\lambda_W^G$ satisfies $\lambda_W^G(xg,y)=\lambda_W^G(x,y)g$,
$\lambda_W^G(x,yg)=\bar g \lambda_W^G(x,y)$, and
$\lambda_W^G(y,x)=\overline{\lambda_W^G(x,y)}$.

Suppose $G\to\Gamma$ is a surjective group homomorphism with
kernel~$H$.  Then the $\Gamma$-cover $W_\Gamma$ is defined, and since
$\pi_1(W) \to G$ induces $\pi_1(W_\Gamma) \to H$, the $H$-cover
$(W_\Gamma)_H$ of $W_\Gamma$ is also defined.  In fact, $W_G =
(W_\Gamma)_H$.  It follows that $H_*(W;RG)\cong H_*(W_\Gamma;RH)$.
Choose pre-images $g_i \in G$ ($i=1,\ldots,|\Gamma|$) of each element
of $\Gamma$, i.e., the $g_i$ are coset representatives of $H\subset
G$.  Note that $\lambda_{W_\Gamma}^H$ is the $RH$-valued intersection
form of $W_\Gamma$.  Regarding $RH$ as a subring of $RG$, we have the
following identity, which can be verified easily by using the defining
formula of $\lambda_W^G$ and~$\lambda_{W_\Gamma}^H$.

\begin{lemma}
  \label{lemma:intersection-of-cover}
  $\displaystyle \lambda_W^G(x,y) = \sum_i
  \lambda_{W_\Gamma}^H(x,yg_i)g_i \in RG$
\end{lemma}

The following observation is also necessary to prove
Theorem~\ref{theorem:covering-solution}.

\begin{lemma}
  \label{lemma:H_2-when-p-torsion-free-H_1}
  Suppose $R$ is a subring of $\Q$.  If $W$ is a 4-manifold with
  boundary $M$ and $H_1(M;R) \to H_1(W;R)$ is surjective, then
  $H_2(W;R)$ is a free $R$-module.
\end{lemma}

\begin{proof}
  We have
  \[
  H_2(W;R)=H^2(W,M;R)=\Hom(H_2(W,M),R) \oplus \Ext(H_1(W,M),R).
  \]
  Since $R$ is torsion free, $\Hom(H_2(W,M),R)$ is isomorphic to a
  free $R$-module of rank~$\beta_2(W)$.  By the surjectivity
  assumption, $0=H_1(W,M;R)=H_1(W,M)\otimes R$.  From this it follows
  that the free part of $H_1(W,M)$ is trivial.  Also, if $H_1(W,M)$
  has a $\Z_{r}$-summand, then $1/r \in R$.  Since
  $\Ext(\Z_{r},R)=R/rR$, it follows that $\Ext(H_1(W,M),R)=0$.
\end{proof}

\begin{proof}[Proof of Theorem~\ref{theorem:covering-solution}]
  Recall our assumptions: $W$ is an $(h)$-solution of $M$, $h>1$,
  $M_\Gamma$ is a $\Gamma$-cover determined by $\pi_1(M)\to \Gamma$,
  and both $H_1(M)$ and $H_1(M_\Gamma)$ are $p$-torsion free.

  Since $H_1(M;\Z_{(p)})\to H_1(W;\Z_{(p)})$ is an isomorphism, $M$
  and $W$ have the same $H_1(-)/(\text{torsion coprime to }p)$.  Since
  $\Gamma$ is an abelian $p$-group, $\pi_1(M) \to \Gamma$ factors
  through $H_1(M)/(\text{torsion coprime to }p)$, and thus it extends
  to $\pi_1(W)$.  Therefore, there is defined the $\Gamma$-cover
  $W_\Gamma$ of $W$ with boundary~$M_\Gamma$.

  Also, $H_1(W,M;\Z_{(p)})=0$.  It follows that
  $H_1(W_\Gamma,M_\Gamma;\Z_{(p)})=0$ by Levine's argument \cite[Proof
  of Proposition 3.2]{Levine:1994-1} (see \cite[Lemma 3.2,
  3.3]{Cha:2007-1}, \cite[Corollary 4.13]{Cochran-Harvey:2007-01} for
  statements that apply to our case directly), and so
  $H_1(M_\Gamma;\Z_{(p)})\to H_1(W_\Gamma;\Z_{(p)})$ is surjective.
  To show that $H_1(M_\Gamma;\Z_{(p)})\to H_1(W_\Gamma;\Z_{(p)})$ is
  an isomorphism, we will prove that in the exact sequence
  \[
  H_2(W_\Gamma;\Z_{(p)}) \to H_2(W_\Gamma,M_\Gamma;\Z_{(p)}) \to
  H_1(M_\Gamma;\Z_{(p)})\to H_1(W_\Gamma;\Z_{(p)}) \to 0
  \]
  the leftmost map is surjective.  For this purpose, we need the
  following two claims:
  
  \begin{claim}
    $\beta_2(W_\Gamma) \le |\Gamma| \beta_2(W)$.
  \end{claim}

  To prove this we appeal to the following fact proved
  in~\cite{Cha:2007-1}: for any $p$-group $\Gamma$,
  \[
  \beta_2(W_\Gamma;\Z_p) \le |\Gamma|
  \beta_2(W;\Z_p)
  \]
  where $\beta_i(-;\Z_p)$ denotes the $\Z_p$-Betti number.  In our
  case, since $H_2(W)$ is $p$-torsion free by
  Lemma~\ref{lemma:H_2-when-p-torsion-free-H_1}, we have
  \[
  \beta_2(W_\Gamma) \le
  \beta_2(W_\Gamma;\Z_p) \le |\Gamma| \beta_2(W;\Z_p) = |\Gamma|
  \beta_2(W).
  \]

  For the second claim, we write $h=n$ or $n.5$ for some integer $n$,
  and let $G=\pi_1(W)/\pi_1(W)^{(n)}$.  Note that there is $G \to
  \Gamma$ since $\pi_1(W)^{(n)} \subset \pi_1(W)^{(1)} \subset
  \pi_1(W_\Gamma) \subset \pi_1(W)$.  As in
  Lemma~\ref{lemma:intersection-of-cover}, let
  $H=\pi_1(W_\Gamma)/\pi_1(W)^{(n)}$ be the kernel of $G\to \Gamma$.

  \begin{claim}
    There are elements $x_1,\ldots,x_m, y_1,\ldots,y_m$ in
    $H_2(W_\Gamma;\Z_{(p)}H)$ such that
    $\lambda_{W_\Gamma}^H(x_i,x_j)=0$ and
    $\lambda_{W_\Gamma}^H(x_i,y_j)=\delta_{ij}$, where $m=\frac12
    |\Gamma| \beta_2(W)$.
  \end{claim}

  \begin{trivlist}
  \item[]
    To construct the $x_i$ and $y_i$, we start with
    $u_1,\ldots,u_r,v_1,\ldots,v_r \in H_2(W;\Z_{(p)}G)$ such that
    $\lambda_W^G(u_i,u_j)=0$ and $\lambda_W^G(u_i,v_j)=\delta_{ij}$,
    which are given by the definition of the $(h)$-solvability, where
    $r=\frac 12 \beta_2(W)$.  Choose a pre-image $g_i\in G$
    ($i=1,\ldots,|\Gamma|$) of each element in~$\Gamma$ as in
    Lemma~\ref{lemma:intersection-of-cover}.  Then the elements
    $u_ig_k$, $v_ig_k$ in $H_2(W;\Z_{(p)}G)=H_2(W_\Gamma;\Z_{(p)}H)$
    satisfy $\lambda_{W_\Gamma}^H(u_ig_k,u_jg_\ell)=0$ and
    $\lambda_{W_\Gamma}^H(u_ig_k,v_jg_\ell)=\delta_{ij}\delta_{k\ell}$
    by Lemma~\ref{lemma:intersection-of-cover}.  This proves Claim~2.
  \end{trivlist}
  
  Returning to the proof of the surjectivity of
  $H_2(W_\Gamma;\Z_{(p)}) \to H_2(W_\Gamma,M_\Gamma;\Z_{(p)})$,
  consider the images $\bar x_i$, $\bar y_i \in
  H_2(W_\Gamma;\Z_{(p)})$ of the $x_i$ and~$y_i$.  By Claim 2 and by
  the naturality of the intersection form, it follows that the
  $\Z_{(p)}$-valued ordinary intersection numbers of the $x_i$ and
  $y_i$ in $W_\Gamma$ are given by $\bar x_i \cdot \bar x_j=0$ and
  $\bar x_i \cdot \bar y_j=\delta_{ij}$.  Therefore $\beta_2(W_\Gamma)
  \ge \frac 12 |\Gamma| \beta_2(W)$.  Combining this with Claim 1, we
  have $\beta_2(W_\Gamma) = \frac 12 |\Gamma| \beta_2(W)$, and
  $H_2(W_\Gamma;\Q)\to H_2(W_\Gamma,M_\Gamma;\Q)$ is an isomorphism.

  Consider the below commutative diagram:
  \[
  \begin{diagram}
    \node{H_1(M_\Gamma;\Z_{(p)})} \arrow{e}\arrow{s}
    \node{H_1(W_\Gamma;\Z_{(p)})} \arrow{s}
    \\
    \node{H_1(M_\Gamma;\Q)} \arrow{e}
    \node{H_1(W_\Gamma;\Q)}
  \end{diagram}
  \]
  The bottom horizontal arrow is injective by the last paragraph, and
  so is the left vertical arrow, since $H_1(M_\Gamma)$ is $p$-torsion
  free.  It follows that the top horizontal arrow is also injective,
  and therefore is an isomorphism.

  Now, observe that $\pi_1(W)^{(n)} \subset \pi_1(W_\Gamma)^{(n-1)}$
  since $\Gamma$ is abelian.  So there is an induced map $H\to
  N=\pi_1(W_\Gamma)/\pi_1(W_\Gamma)^{(n-1)}$.  Let $x_i',y_i' \in
  H_2(W_\Gamma;\Z_{(p)}N)$ be the image of $x_i,y_i\in
  H_2(W_\Gamma;\Z_{(p)}H)$ given in Claim~2 ($i=1,\ldots,m$).  Then by
  the naturality of the intersection form,
  $\lambda_{W_\Gamma}^N(x'_i,x'_i)=0$ and
  $\lambda_{W_\Gamma}^N(x'_i,y'_i)=\delta_{ij}$.  Since $m=\frac12
  |\Gamma| \beta_2(W) = \frac12 \beta_2(W_\Gamma)$, it follows that
  $W_\Gamma$ is an $(h-1)$-solution of $M_\Gamma$ when $h=n$.

  For the case of $h=n.5$, there are $\tilde u_1,\ldots,\tilde u_r \in
  H_2(W;\Z_{(p)}[\pi_1(W)/\pi_1(W)^{(n+1)}])$ such that
  $\lambda_W^{(n+1)}(\tilde u_i,\tilde u_j)=0$ and $u_i$ is the image
  of $\tilde u_i$.  Similarly to the construction of the above $x_i$,
  we can produce $m$ elements $\tilde x_i \in
  H_2(W_\Gamma;\Z_{(p)}[\pi_1(W_\Gamma)/\pi_1(W_\Gamma)^{(n)}])$ from
  the $\tilde u_i$, in such a way that the $\tilde x_i$ together with
  the $x_i, y_i$ satisfy the intersection form condition required in
  the definition of $(n.5)$-solvability.  This completes the proof for
  $h=n.5$.
\end{proof}

\subsection{Vanishing of $\lambda(M,\phi)$ for $(0.5)$-solvable $M$}

\begin{theorem}
  \label{theorem:0.5-solvability-obstruction}
  Suppose $M$ is $(0.5)$-solvable 3-manifold with $p$-torsion free
  $H_1(M)$ and $\phi\colon\pi_1(M) \to \Gamma=\Z_{p^a}$ is a character
  such that $\beta_1(M_{\Gamma_i})-1=p^i(\beta_1(M)-1)$ for
  $i=0,1,\ldots,a$.  Then $\lambda(M,\phi)=0$ in
  $L^0(\Q(\zeta_{p^a}))$.
\end{theorem}

Here $M_{\Gamma_i}$ is the $\Z_{p^i}$-cover of $M$ defined in the
paragraph before Theorem~\ref{theorem:n.5-solvaility-obstruction}.

By applying Covering Solution Theorem~\ref{theorem:covering-solution}
inductively $n$ times and applying
Theorem~\ref{theorem:0.5-solvability-obstruction} to the $n$th
iterated cover, the main theorem of this section
(Theorem~\ref{theorem:n.5-solvaility-obstruction}) follows
immediately.

\begin{proof}
  \setcounter{claim}{0}%
  Suppose $W$ is a ($\Z_{(p)}$-coefficient) $(0.5)$-solution of $M$.
  Since $H_1(M;\Z_{(p)})\cong H_1(W;\Z_{(p)})$, $\phi$ extends to
  $\pi_1(W)$ as in the proof of
  Theorem~\ref{theorem:covering-solution}.  Therefore
  $\lambda(M,\phi)$ is defined and can be computed from the
  intersection forms of~$W$.

  Since $W$ is a $(0.5)$-solution, there are $\tilde u_1,\ldots,\tilde
  u_r\in H_2(W;\Z_{(p)}[\pi/\pi^{(1)}])$ and $v_1,\ldots,u_r\in
  H_2(W;\Z_{(p)})$ such that $\lambda^{(1)}_W(\tilde u_i, \tilde
  u_j)=0$ and $\lambda_W(u_i,v_j)=\delta_{ij}$ where $\pi=\pi_1(W)$,
  $r=\frac12 \beta_2(W)$, and $u_i$ is the image of~$\tilde u_i$.
  Multiplying the $\tilde u_i$ and $v_j$ by some constant $c$ coprime
  to $p$, we may assume that the $u_i$ and $v_i$ are in the image of
  $H_2(W;\Z[\pi/\pi^{(1)}])$ and $H_2(W)$, at the cost of
  $\lambda_W(u_i,v_j)=c^2\delta_{ij}$ instead of $\delta_{ij}$.
  Choose the pre-images $\tilde x_i \in
  H_2(W;\Z_{(p)}[\pi/\pi^{(1)}])$ of $\tilde u_i$ and $y_j\in
  H_2(W;\Z_{(p)})$ of $v_j$, respectively.

  Denote $\omega_i=\zeta_{p^i}$.  Since $\Gamma_i$ is abelian, there
  is a canonically induced map $\Z[\pi/\pi^{(1)}] \to \Z\Gamma_i$, and
  $\Z\Gamma_i$, $\Z_p\Gamma_i$, $\Q\Gamma_i$, and $\Q(\omega_i)$ can
  be used as homology coefficients of $W$ (and $M$, $(W,M)$ as well).
  For $\sR= \Z\Gamma_i$, $\Z_p\Gamma_i$, $\Q\Gamma_i$, and
  $\Q(\omega_i)$, let
  \begin{align*}
    A(\sR) &= \Im\{H_2(W;\sR) \to H_2(W,M;\sR)\} \text{, and}\\
    L(\sR) &= \text{$\sR$-submodule of $A(\sR)$ generated by the image
      of~$\tilde x_i$.}
  \end{align*}
  Note that $\Gamma_0$ is the trivial group so that $\Z\Gamma_0=\Z$
  and similarly for $\Z_p$ and $\Q$ coefficients.
  
  \begin{claim}
    $\dim_\Q L(\Q\Gamma_i) \ge \frac12 \dim_\Q A(\Q\Gamma_i)$ for
    $i=0,1,\ldots,a$.
  \end{claim}
  
  To prove Claim 1, we need the following fact, which is a
  $p$-analogue of a result due to Cochran and
  Harvey~\cite[Proposition~2.7]{Cochran-Harvey:2006-01}:
  \[
  \dim_{\Z_p} \frac{H_2(W,M;\Z_p\Gamma_i)}{L(\Z_p\Gamma_i)} \le
  |\Gamma_i| \dim_{\Z_p} \frac{H_2(W,M;\Z_p)}{L(\Z_p)}
  \tag{$*$}
  \]
  Since the proof of \cite[Proposition~2.7]{Cochran-Harvey:2006-01}
  carries over to our $p$-group case (using $p$-analogous statements
  when necessary; see also~\cite{Cochran-Harvey:2007-01}), we omit
  details of the proof of~$(*)$.

  Since $H_1(M;\Z_{(p)}) \cong H_1(W;\Z_{(p)})$,
  $H_1(W,M;\Z_{(p)})=0$.  Consequently $H_1(W,M)$ is $p$-torsion free.
  Also, observe that $H_1(W)$ is $p$-torsion free since
  $H_1(W;\Z_{(p)})\cong H_1(M;\Z_{(p)})$ has no torsion by the
  hypothesis.  So,
  \begin{align*}
    H_2(W,M;\Z_{(p)})&=H^2(W;\Z_{(p)})=\Hom(H_2(W),\Z_{(p)})\oplus
    \Ext(H_1(W),\Z_{(p)})\\
    &=\Hom(H_2(W),\Z_{(p)})
  \end{align*}
  has no torsion, and therefore $H_2(W,M)$ is $p$-torsion free.  From
  these observations, it follows that
  \[
  \dim_{\Z_p} H_2(W,M;\Z_p)=\rank_\Z H_2(W,M)/\text{torsion}=\beta_2(W).
  \]

  Recall that $L(\Z_p)$ is generated by the images of the
  $\frac12\beta_2(W)$ elements $\tilde x_i$.  Since there are the dual
  elements, namely the images of the $y_j$ whose intersection with the
  image $\tilde x_i$ is $c^2\delta_{ij}$, and since our constant $c$ is
  coprime to $p$, it follows that the images of the $\tilde x_i$ are
  $\Z_p$-linearly independent in $H_2(W,M;\Z_p)$ so that $\dim_{\Z_p}
  L(\Z_p) = \frac12\beta_2(W)$.  Therefore, by $(*)$, we have
  \[
  \dim_{\Z_p} \frac{H_2(W,M;\Z_p\Gamma_i)}{L(\Z_p\Gamma_i)} \le
  \frac{|\Gamma_i|}{2} \beta_2(W).
  \tag{$**$}
  \]

  By Levine's result \cite[Proof of Proposition 3.2]{Levine:1994-1}
  (refer to \cite[Lemma 3.2 and 3.3]{Cha:2007-1}, \cite[Corollary
  4.13]{Cochran-Harvey:2007-01} for statements that apply to our case
  directly), $H_1(W,M;\Z_{(p)})=0$ implies that
  $H_1(W,M;\Z_{(p)}\Gamma_i)=0$.  So $H_1(W,M;\Z\Gamma_i)$ is
  $p$-torsion free, and
  \[
  H_2(W,M;\Z_p\Gamma)=H_2(W,M;\Z\Gamma)\otimes_\Z \Z_p.
  \]
  Since ${}\otimes_\Z\Z_p$ preserves cokernels, we have
  \begin{align*}
    \frac{H_2(W,M;\Z_p\Gamma_i)}{L(\Z_p\Gamma_i)} &=
    \Coker\{L(\Z\Gamma_i) \otimes_\Z \Z_p \to H_2(W,M;\Z\Gamma_i)
    \otimes_\Z \Z_p \} \\
    &= \Coker\{L(\Z\Gamma_i) \to H_2(W,M;\Z\Gamma_i) \} \otimes_\Z \Z_p \\
    &= \frac{H_2(W,M;\Z\Gamma_i)}{L(\Z\Gamma_i)} \otimes_\Z \Z_p
  \end{align*}
  and similarly
  \[
  \frac{H_2(W,M;\Q\Gamma_i)}{L(\Q\Gamma_i)}
  = \frac{H_2(W,M;\Z\Gamma_i)}{L(\Z\Gamma_i)} \otimes_\Z \Q.
  \]
  By $(**)$, it follows that
  \[
  \dim_{\Q} \frac{H_2(W,M;\Q\Gamma_i)}{L(\Q\Gamma_i)} \le
  \dim_{\Z_p} \frac{H_2(W,M;\Z_p\Gamma_i)}{L(\Z_p\Gamma_i)}
  \le \frac{|\Gamma_i|}{2} \beta_2(W).
  \tag{$*\mathord{*}*$}
  \]
  Since $H_1(W,M;\Z_{(p)}\Gamma_i)=0$ as shown above, we have an exact
  sequence
  \[
  0 \to A(\Q\Gamma_i) \to H_2(W,M;\Q\Gamma_i) \to H_1(M;\Q\Gamma_i)
  \to H_1(W;\Q\Gamma_i) \to 0
  \]
  and from this it follows that
  \[
  \dim_\Q A(\Q\Gamma_i) = \beta_2(W_{\Gamma_i}) -
  \beta_1(M_{\Gamma_i}) + \beta_1(W_{\Gamma_i}).
  \]
  Note that $\beta_1(M)=\beta_1(W)$, and by our hypothesis,
  $\beta_1(M_{\Gamma_i})=|\Gamma_i|(\beta_1(M)-1)+1$.  Combining this
  with $(*\mathord{*}*)$, we obtain
  \begin{multline*}
    2\dim_\Q L(\Q\Gamma_i) -\dim_\Q A(\Q\Gamma_i) \\
    \begin{aligned}[t]
      &\ge (\beta_2(W_{\Gamma_i})-\beta_1(W_{\Gamma_i})+1) - |\Gamma_i|
      (\beta_2(W)-\beta_1(W)+1) \\
      &= \chi(W_{\Gamma_i})-|\Gamma_i| \chi(W) = 0
    \end{aligned}
  \end{multline*}
  since $\beta_3(W)=\beta_1(W,M)=0$ and $\beta_3(W_{\Gamma_i})=0$
  similarly.  This finishes the proof of Claim~1.
   
  \begin{claim}
    $ \dim_{\Q(\omega_i)} L(\Q(\omega_i)) = \frac12
    \dim_{\Q(\omega_i)} A(\Q(\omega_i)) $ for $i=0,1,\ldots,a$.
  \end{claim}

  To prove this, we use an induction on $i$.  First observe that $\le$
  always holds since the $\Q(\omega_i)$-valued intersection form on
  $A(\Q(\omega_i))$ is nonsingular but vanishes on $L(\Q(\omega_i))$.
  So it suffices to show~$\ge$.  For $i=0$, this is none more than
  Claim~1 for $i=0$.

  Suppose $k>0$ and that Claim~2 holds for all $i < k$.  Recall that
  the regular representation $\Q[\Z_d]$ is isomorphic to the
  orthogonal sum $\bigoplus_{r|d} \Q(\zeta_r)$.  It follows that for
  any $j<i$, $\Q(\omega_j)$ are $(\Q\Gamma_i)$-flat.  So, we have
  \begin{align*}
    H_2(W;\Q(\omega_j)) &= H_2(W;\Q\Gamma_i) \otimes_{\Q\Gamma_i} 
    \Q(\omega_j), \\
    H_2(W,M;\Q(\omega_j)) &= H_2(W,M;\Q\Gamma_i) \otimes_{\Q\Gamma_i}
    \Q(\omega_j),
  \end{align*}
  and therefore
  \[
  A(\Q(\omega_j)) = A(\Q\Gamma_i) \otimes_{\Q\Gamma_i} \Q(\omega_j)
  \quad\text{and}\quad
  L(\Q(\omega_j)) = L(\Q\Gamma_i) \otimes_{\Q\Gamma_i} \Q(\omega_j).
  \]
  Using this, we have
  {
    \def\ot{\mathbin{\mathop{\otimes}\limits_{\Q\Gamma_k}}}
    \begin{align*}
      L(\Q\Gamma_k) &= L(\Q\Gamma_k) \ot \Q\Gamma_i = L(\Q\Gamma_k)
      \ot \bigg(\bigoplus_{j=0}^k \Q(\omega_i) \bigg) \\
      &= \bigoplus_{j=0}^k L(\Q\Gamma_k) \ot \Q(\omega_i)
      = \bigoplus_{j=0}^k L(\Q(\omega_j))
    \end{align*} 
  }
  and so
  \[
  \dim_\Q L(\Q\Gamma_k) = \sum_{j=0}^{k} [\Q(\omega_j):\Q]
  \dim_{\Q(\omega_j)} L(\Q(\omega_j)).
  \]
  Similarly, we have an analogous formula for $\dim_\Q A(\Q\Gamma_k)$,
  which is obtained by replacing every occurrence of $L(-)$ with
  $A(-)$.  By Claim~1 for $i=k$ and by the induction hypothesis, we
  have
  \[
  \dim_{\Q(\omega_k)} L(\Q(\omega_k)) \ge \frac12 \dim_{\Q(\omega_k)}
  A(\Q(\omega_k))
  \]
  as desired.  This completes the proof of Claim~2.
  
  Now, the nonsingular part of the $\Q(\omega_i)$-valued intersection
  form $\lambda_{\Q(\omega_i)}(W)$, which is defined on $A(\Q(w_i))$,
  vanishes on a half-dimensional subspace, namely $L(\Q(w_i))$.  It
  follows that the Witt class $[\lambda_{\Q(\omega_i)}(W)]$ is trivial
  in $L^0(\Q(\omega_i))$ for $i=0,1,\ldots,a$.  Denoting the natural
  map $L^0(\Q) \to L^0(\Q(\zeta_{p^a}))$ by $i_*$, by definition we
  have
  \[
  \lambda(M,\phi)=[\lambda_{\Q(\zeta_{p^a})}(W)]-i_*[\lambda_\Q(W)]=0.
  \qedhere
  \]
\end{proof}

\section{Local knots}
\label{section:local-knots}

In order to investigate the structure of link concordance modulo knot
concordance, we consider the quotient of string link concordance group
by ``local knots''; Figure~\ref{figure:local-knot} illustrates the
effect of adding a local knot to a string link.

\begin{figure}[ht]
  \begin{center}
    \includegraphics{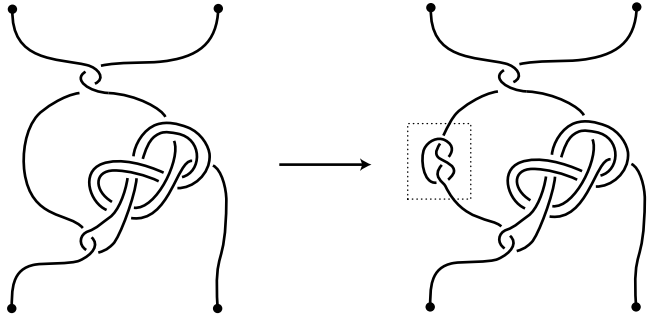}
    \caption{Adding a local knot}
    \label{figure:local-knot}
  \end{center} 
\end{figure}

It can be seen easily that adding a local knot to an $m$-string link
is equivalent to multiplication by a string link $\beta$ which is the
completely split union of one (possibly knotted) arc and $m-1$
unknotted line segments.  Possibly as a slight abuse of terminology,
we call $\beta$ a \emph{local knot}.  Note that a local knot $\beta$
commutes with any string link, so that any subgroup generated by a
collection of local knots is normal.

In this section we show that if a $p$-structure $\T$ is locally
trivial, our invariant $\lambda_\T$ ignores the effect of
adding/removing local knots.  Recall from the introduction that
$\T=(\{X_i\},\phi)$ of height~$n$ is called \emph{locally trivial} if
$\phi(z)=0$ whenever $z$ is an element in $\pi_1(X_n)$ which projects
to a conjugate of $[x_i]^r$ in $\pi_1(X)$ for some~$r$.  (Here $x_i$
denotes the $i$th circle of~$X$.)

\begin{remark}
  \label{remark:alternative-description-of-local-triviality}
  We often view a character (into an abelian group) as a homomorphism
  of $H_1$, so that $\theta([\gamma])$ is well-defined for the
  homology class $[\gamma]$ of a loop $\gamma$ based at any point
  in~$X_n$.  Then it can be seen easily that $\T$~is locally trivial
  if and only if $\theta([\gamma])=0$ whenever $\gamma$ is a loop in
  $X_n$ whose projection in $X$ is freely homotopic to $x_i^r$ for
  some~$r$.  Moreover, since a free homotopy in $X$ lifts to $X_n$ and
  $\theta$ is preserved by free homotopy, ``freely homotopic to'' in
  the above statement can be omitted.
\end{remark}

\begin{theorem}
  \label{theorem:vanishing-for-local-knots}
  If a $p$-structure $\T$ for $X$ is locally trivial, then
  $\lambda_\T(\beta)=0$ for any local knot~$\beta$.
\end{theorem}

As an immediate consequence of
Theorems~\ref{theorem:additivity-of-lambda_T},
\ref{theorem:n.5-solvaility-obstruction}, and
\ref{theorem:vanishing-for-local-knots}, we obtain Addendum to
Theorem~\ref{theorem:lambda_T-homomorophism} and the second part of
Theorem~\ref{theorem:COT-filtration-and-lambda_T}: If $\T$ is locally
trivial, then $\lambda_\T$ induces a homomorphism
\[
\lambda_\T\colon \frac{\hCSL}{\langle\text{local knots}\rangle} \to
L^0(\Q(\zeta_d)).
\]
If, in addition, $\T$ is of height $n$, then $\lambda_\T$ induces a
homomorphism
\[
\lambda_\T\colon \frac{\hFSL_{(n)}}{\hFSL_{(n.5)} \cdot
  \langle\text{local knots}\rangle} \subset \frac{\hCSL}{\hFSL_{(n.5)}
  \cdot \langle\text{local knots}\rangle} \to L^0(\Q(\zeta_d)).
\]

To prove Theorem~\ref{theorem:vanishing-for-local-knots}, we appeal to
the following result, which will also be used in later sections as our
main computation tool for 3-manifolds and (string) links which are
obtained by \emph{infection}.  Suppose $M$ is a 3-manifold and
$\alpha$ is a simple closed curve in~$M$.  Removing an open tubular
neighborhood of $\alpha$ from $M$ and then filling it in with the
exterior of a knot $K$ in $S^3$ so that the meridian and preferred
longitude of $K$ is identified with a longitude and meridian of
$\alpha$, respectively, we obtain a new 3-manifold~$N$.  We say that
$N$ is obtained from $M$ by infection along $\alpha$ using~$K$.
In~\cite[Proposition~4.8]{Cha:2007-1}, it was shown that there is a
$p$-tower map $N\to M$ which extends the identity map between
$N-($exterior of $K)$ and $M-($tubular neighborhood of $\alpha)$.

\begin{proposition}[Corollary 4.7 of \cite{Cha:2007-1}]
  \label{proposition:lambda-of-infected-manifold}
  Suppose $N$ is obtained from $M$ by infection along a simple closed
  curve $\alpha$ using a knot~$K$.  Suppose $(\{M_i\},\phi)$ is a
  $\Z_d$-valued $p$-structure of height $n$ and let $(\{N_i\},\psi)$
  be the induced $p$-structure via pullback along the $p$-tower map
  $N\to M$.  Let $\tilde\alpha_1, \tilde\alpha_2,\ldots \subset M_n$
  be the components of the pre-image of $\alpha\subset M$, and $r_j$
  be the degree of the covering map $\tilde\alpha_j \to \alpha$.  Then
  \[  
  \lambda(N_n,\psi)=\lambda(M_n,\phi)+\sum_{j} \Big(
  [\lambda_{r_{j}}(A,\zeta_d^{\phi([\tilde\alpha_j])})]-[\lambda_{r_{j}}(A,1)]
  \Big)
  \]
  where $[\lambda_{r}(A,\omega)]$ is the Witt class of (the
  nonsingular part of) the hermitian form represented by the following
  $r\times r$ block matrix:
  \[
  \lambda_{r}(A,\omega) =
  \begin{bmatrix}
    \vphantom{\ddots} A+A^T & -A & & -\omega^{-1} A^T\\
    \vphantom{\ddots} -A^T & A+A^T & \ddots \\
    \vphantom{\ddots} & \ddots & \ddots & -A \\
    \vphantom{\ddots} -\omega A & & -A^T & A+A^T
  \end{bmatrix}_{r\times r}
  \]
  For $r=1$ and $2$, $\lambda_{r}(A,\omega)$ should be understood as
  \[
  \begin{bmatrix}
    (1-\omega)A+(1-\omega^{-1})A^T
  \end{bmatrix}
  \quad\text{and}\quad
  \begin{bmatrix}
    A+A^T & -A-\omega^{-1}A^T \\
    -A^T-\omega A & A+A^T
  \end{bmatrix}.
  \]
\end{proposition}

Performing infection on the surgery manifold $M_\beta$ of a string
link $\beta$ along an unknotted curve $\alpha$ in
$D^2\times[0,1]-\beta$ in such a way that the meridian of $K$ is
identified with the preferred longitude of $\alpha$, the resulting
manifold is the surgery manifold $M_\beta'$ of a new string link
$\beta'$, which is said to be obtained from $\beta$ by infection.
Obviously, the composition of the preferred meridian map $X\to
M_{\beta'}$ of $\beta'$ and the $p$-tower map $M_{\beta'} \to
M_{\beta}$ is exactly the preferred meridian map of~$\beta$.  So we
have a formula relating $\lambda_\T(\beta')$ to $\lambda_\T(\beta)$
and invariants of $K$, similarly to
Proposition~\ref{proposition:lambda-of-infected-manifold}.

\begin{proof}[Proof of Theorem~\ref{theorem:vanishing-for-local-knots}]
  Suppose $\T=(\{X_i\},\phi)$ is a $\Z_d$-valued $p$-structure of
  height $n$ for~$X$.  Observe that a local knot $\beta$ is obtained
  from the trivial string link by infection along the $i$th meridian
  for some~$i$.  Since the trivial link has vanishing $\lambda_\T$, by
  Proposition~\ref{proposition:lambda-of-infected-manifold} we have
  \[
  \lambda_\T(\beta) = \sum_j \Big(
  [\lambda_{r_j}(A,\zeta_d^{\phi([\tilde\alpha_j])})] -
  [\lambda_{r_j}(A,1)] \Big)
  \]
  where $\tilde\alpha_j$ is a loop in $X_n$ which is a lift
  of~$x_i^{r_j}$ and $A$ is a Seifert matrix of the infection knot.
  Since $\T$ is locally trivial, $\phi([\tilde\alpha_j]) = 0$ for
  all~$j$.  It follows that $\lambda_\T(\beta)=0$.
\end{proof}

For later use, we observe that the defining condition of a locally
trivial $p$-structure is preserved by pullback, as stated below:

\begin{lemma}
  \label{lemma:local-triviality-under-pullback}
  Suppose $f\colon X \to Y$ is a map into a CW-complex $Y$ and
  $\T=(\{X_k\}, \theta)$ is a $p$-structure for $X$ induced by a
  $p$-structure $(\{Y_k\},\phi)$ of height $n$ for $Y$ via pullback
  along~$f$.  Then $\T$ is locally trivial if and only if
  $\phi([\delta])=0$ whenever $\delta$ is a loop in $Y_n$ whose
  projection in $Y$ is freely homotopic to $f(x_i^r)$ for some~$r$.
\end{lemma}

\begin{proof}
  Let $f_n\colon X_n \to Y_n$ be the lift of~$f$.  Note that a loop
  $\delta$ in $Y_n$ projects to $f(x_i^r)$ if and only if
  $\delta=f_n(\gamma)$ for some a loop $\gamma$ in $X_n$ which
  projects to~$x_i^r$.  (For the only if part we need that $X_n$ is
  the pullback of~$Y_n$.)  Also, if a loop $\delta$ in $Y_n$ is freely
  homotopic to $f_n(\gamma)$ for some loop $\gamma$ in $X_n$, then
  $\phi([\delta])=\theta([\gamma])$ since $\theta$ is induced
  by~$\phi$.  The conclusion follows from this.
\end{proof}

\section{Construction of examples}
\label{section:examples}

In \cite{Cha:2007-1}, it was proved that the $n$th iterated Bing
doubles (which has $2^n$ components) of certain knots are nontrivial
in $\FL_{(n)}/\FL_{(n+1)}$ but have vanishing Harvey's
$\rho_n$-invariant.  In this section, for any $m$ we construct
infinitely many $m$-component string links with similar properties,
and furthermore we show that they are linearly independent in the
abelianization of $\FL_{(n)}/\FL_{(n.5)}$.

% Our examples are constructed by ``infection'' on a trivial string link
% $\beta_0$ along a certain unknotted simple closed curve $\alpha$ in
% $D^2\times[0,1]$ which is disjoint from~$\beta_0$.  Namely, removing
% from $D^2\times[0,1]-\beta$ an open tubular neighborhood of $\alpha$
% and then filling it in with the exterior of a knot $K$ so that the
% preferred longitude and meridian of $K$ is identified with the
% meridian and preferred longitude of $\alpha$, respectively, we obtain
% the complement of a new string link, say~$\beta$.  We say that $\beta$
% is obtained from $\beta_0$ via infection by $K$ along~$\alpha$.
% Roughly speaking, $\beta$ is obtained from $\beta_0$ by ``tying'' $K$
% along a disk in $D^2\times[0,1]$ bounded by~$\alpha$.  Infection on a
% link in $S^3$ is defined similarly.

Our examples are constructed by infection on a trivial string link.
First we show that the class of $\hF$-(string) links is closed under
infection.

\begin{proposition}
  \label{proposition:infection-on-hat-F-link}
  Any (string) link obtained from an $\hF$-(string) link via
  infection by a knot is an $\hF$-(string) link.
\end{proposition}

\begin{proof}
  It suffices to prove the conclusion in the case of a link.  Suppose
  $L'$ is obtained from a link $L$ via by infection along $\alpha$
  using a knot.  Let $E$ and $E'$ be the exteriors of $L$ and $L'$,
  respectively.  The $p$-tower map $f\colon E' \to E$ given by
  \cite[Proposition~4.8]{Cha:2007-1} is actually an integral homology
  equivalence preserving the peripheral structure (see the proof of
  \cite[Proposition~4.8]{Cha:2007-1}).  Choose a meridian map
  $\mu\colon X=\bigvee^m S^1 \to E$, i.e., $\mu$ sends the $i$th
  circle of $X$ to a $i$th meridian of~$L$.  We may assume that the
  image of $\mu$ is disjoint from a tubular neighborhood of $\alpha$
  so that $\mu$ gives rise to a meridian map $\mu'\colon X \to E'$
  satisfying $f\circ \mu' = \mu$.  Since $L$ is an $\hF$-link, $\mu$
  induces an isomorphism on $\widehat{\pi_1(-)}$~\cite{Levine:1989-1}.
  Also, so does $f$ by
  Lemma~\ref{lemma:algebraic-closure-and-p-tower-map}~(1), since $f$
  is a homology equivalence.  It follows that $\mu'$ induces an
  isomorphism on $\widehat{\pi_1(-)}$.  Since the longitudes of $L$
  are killed in $\widehat{\pi_1(E)}$ and $f$ preserves the peripheral
  structure, the longitudes of $L'$ is killed in $\widehat{\pi_1(E')}
  \cong \widehat{\pi_1(E)}$.  Therefore $L'$ is an $\hF$-link.
\end{proof}

We will use infection knots and curves obtained by applying the
following two lemmas.  In the first lemma, we denote the
Levine-Tristram signature of $K$ by $\sigma_K(\omega)$, i.e., for
$\omega\in S^1 \subset \C$, $\sigma_K(\omega)$ is defined to be the
average of the two one-sided limits of the \emph{$\omega$-signature of
  $K$} given by
\[
\sign[ (1-\omega)A + (1-\omega^{-1})A^T]
\]
where $A$ is a Seifert matrix of~$K$.  We remark two facts: first, one
should think of the average in order to obtain a concordance invariant
for any $\omega\in S^1$.  Second, when $\omega$ is a \emph{primitive}
$p^a$th root of unity for some prime power $p^a$, it is known that
there is no nontrivial jump of the $\omega$-signature at $\omega$, so
that the average $\sigma_K(\omega)$ is equal to the
$\omega$-signature.  In this case we do not distinguish
$\sigma_K(\omega)$ and the $\omega$-signature.

\begin{lemma}
  \label{lemma:knots-with-independent-signatures}
  There is an infinite sequence $\{K_i\}$ of knots together with a
  strictly increasing sequence $\{d_i\}$ of powers of $p$ satisfying
  the following properties:
  \begin{enumerate}
  \item $\sigma_{K_i}(\zeta_{d_i}) > 0$, and if $p=2$ then
    $\sigma_{K_i}(\zeta_{d_i}^s) \ge 0$ for any~$s$.
  \item if $i<j$, $\sigma_{K_j}(\zeta_{d_i}^s) = 0$ for any $s$.
  \item $\int_{S^1} \sigma_{K_i}(\omega) \, d\omega = 0$.
  \item $K_i$ has vanishing Arf invariant.
  \end{enumerate}
\end{lemma}

We remark that the conclusion of
Lemma~\ref{lemma:knots-with-independent-signatures} is slightly
stronger than what we actually need in this section; the extra parts
will be used in a later section.

\begin{lemma}
  \label{lemma:infection-curve-and-tower}
  For any $m>1$ and $n$, there exist a loop $\alpha$ in $X=\bigvee^m
  S^1$ and a $p$-tower $\{X_k\}$ of height $n$ for $X$ satisfying the
  following:
  \begin{enumerate}
  \item $[\alpha] \in \pi_1(X)^{(n)}$ and every lift $\tilde \alpha_j$
    of $\alpha$ in $X_n$ is a loop.
  \item There is a map $f\colon \pi_1(X_n) \to \Z$ which sends (the
    class of) each $\tilde\alpha_j$ to $-1$, $0$, or $1$ and sends at
    least one $\tilde\alpha_j$ to~$1$.  In addition, for any $\Z \to
    \Z_d$, the composition $\theta\colon \pi_1(X_n) \xrightarrow{f} \Z
    \to \Z_d$ gives a locally trivial $p$-structure
    $(\{X_i\},\theta)$.
%   \item For any $d>0$ and $0 \le s < d$, there is a character
%     $\phi_{s,d}\colon\pi_1(X_n) \to \Z_d$ such that
%     $\phi_{s,d}([\tilde\alpha_j])=-s$, $0$, or $s$ for any $\tilde
%     \alpha_j$, and $\phi_{s,d}([\tilde\alpha_j])=s$ for at least
%     one~$\tilde \alpha_j$.
  \end{enumerate}
\end{lemma}

We postpone the proof of
Lemmas~\ref{lemma:knots-with-independent-signatures}
and~\ref{lemma:infection-curve-and-tower}, and proceed to discuss how
our examples are constructed.

Suppose $\{K_i\}$ satisfies
Lemma~\ref{lemma:knots-with-independent-signatures} and $\alpha$
satisfies Lemma~\ref{lemma:infection-curve-and-tower}.  Let $\mu\colon
X \to E_{\beta_0}$ be the preferred meridian map of a trivial string
link $\beta_0$.  Choose any simple closed curve in $E_{\beta_0}$ which
is unknotted in $D^2\times[0,1]$ and realizes the homotopy class of
$\mu_*([\alpha]) \in \pi_1(E_{\beta_0})$.  Let $\beta(K_i)$ be the
string link obtained from $\beta_0$ via infection using $K_i$ along
the chosen simple closed curve.  Note that $\beta(K_i)$ is always an
$\hF$-string link by
Proposition~\ref{proposition:infection-on-hat-F-link}.

\begin{theorem}
  \label{theorem:Z-independence-of-string-links}
  The string links $\beta(K_i)$ are $(n)$-solvable and linearly
  independent in the abelianization of
  $\hCSL\big/(\hFSL_{(n.5)}\cdot\langle\text{local knots}\rangle)$.
  Consequently, the abelianization of
  $\hFSL_{(n)}\big/(\hFSL_{(n.5)}\cdot\langle\text{local
    knots}\rangle)$ is of infinite rank for any~$n$.
\end{theorem}

\begin{proof}
  $\beta(K_i)$ is an $\hF$-string link by
  Proposition~\ref{proposition:infection-on-hat-F-link}.  Since
  $[\alpha]\in \pi_1(X)^{(n)}$ by
  Lemma~\ref{lemma:infection-curve-and-tower}~(1) and $K_i$ has
  vanishing Arf invariant by
  Lemma~\ref{lemma:knots-with-independent-signatures}~(4),
  $\beta(K_i)$ is $(n)$-solvable by
  \cite{Cochran-Orr-Teichner:1999-1}.  So (the concordance class of)
  $\beta(K_i)$ is in~$\hFSL_{(n)}$.

  In order to show the independence, we use the invariant
  $\lambda_\T$.  Let $\{X_k\}$ be the $p$-tower in
  Lemma~\ref{lemma:infection-curve-and-tower}, and let
  $\theta_{d}\colon\pi_1(X_n) \to \Z_d$ be the composition of the map
  $f\colon \pi_1(X_n) \to \Z$ in
  Lemma~\ref{lemma:infection-curve-and-tower} and the projection
  $\Z\to \Z_d$ sending $1\in\Z$ to $1\in\Z_d$.  Then, for
  $\T=(\{X_k\},\theta_{d})$ and a knot $K$, appealing to
  Proposition~\ref{proposition:lambda-of-infected-manifold} we have
  \[
  \lambda_\T(\beta(K)) = \sum_j \big(
  [\lambda_{1}(A,\zeta_d^{\theta_{d}([\tilde\alpha_j])})] -
  [\lambda_{1}(A,1)] \big)
  \]
  where $A$ is a Seifert matrix of~$K$.  Observe that $\sign
  \lambda_1(A,\omega)=\sigma_{K}(\omega)$, $\sigma_{K}(1)=0$, and
  $\sigma_{K}(\omega) = \sigma_{K}(\omega^{-1})$.  So by
  Lemma~\ref{lemma:infection-curve-and-tower}~(2),
  \[
  \sign \lambda_\T(\beta(K)) = c\cdot \sigma_{K}(\zeta_d).
  \]
  where $c$ is the number of lifts $\tilde\alpha_j$ sent to $\pm 1$ by
  $f\colon \pi_1(X_n) \to \Z$.  Note that $c>0$ and $c$ is independent
  of~$d$.

  Suppose $\sum_i a_i \beta(K_i)=0$ in the abelianization of
  $\hCSL/(\hFSL_{(n.5)}\cdot\langle\text{local knots}\rangle)$ where
  not all $a_i$ are zero.  Choose a minimal $i_0$ such that
  $a_{i_0}\ne 0$.  Let $d_i$ be as in
  Lemma~\ref{lemma:knots-with-independent-signatures}, and let
  $\T=(\{X_k\},\phi_{d_{i_0}})$.  Note that $\T$ is locally trivial by
  Lemma~\ref{lemma:infection-curve-and-tower}~(2).  Then by
  Theorem~\ref{theorem:COT-filtration-and-lambda_T},
  Lemma~\ref{lemma:knots-with-independent-signatures}, and by our
  choice of $i_0$, we have
  \begin{align*}
    0 = \sign \lambda_\T\Big(\sum_i a_i \beta(K_i)\Big) & = \sum_i
    a_i\sign\lambda_\T( \beta(K_i)) \\
    &= \sum_{i\ge i_0} c a_i\cdot \sigma_{K_i}(\zeta_{d_{i_0}}) = c
    a_{i_0}\cdot \sigma_{K_{i_0}}( \zeta_{d_{i_0}}).
  \end{align*}
  Since $c$ and $\sigma_{K_{i_0}}(\zeta_{d_{i_0}})$ is nonzero,
  $a_{i_0}$ should be zero.  From this contradiction, it follows that
  the $\beta(K_i)$ are linearly independent in the abelianization of
  $\hCSL/(\hFSL_{(n.5)}\cdot\langle\text{local knots}\rangle)$.
\end{proof}

Since $\int_{S^1} \sigma_{K_i}(\omega) \, d\omega=0$ by
Lemma~\ref{lemma:knots-with-independent-signatures}~(3), it follows
that each $\beta(K_i)$ has vanishing Harvey's $\rho_n$-invariant by
results in~\cite{Harvey:2006-1}.  Also, note that each $\beta(K_i)$ is
a boundary link.  (In fact it can be seen that any link obtained from
a boundary link by infection is again a boundary link.)  So, as an
immediate consequence of
Theorem~\ref{theorem:Z-independence-of-string-links}, we obtain
Theorem~\ref{theorem:kernel-of-harvey-invariant}: the abelianization of
the kernel of Harvey's homomorphism
\[
\rho_n\colon \frac{\BFSL_{(n)}}{\BFSL_{(n.5)}\cdot\langle\text{local
    knots}\rangle} \to \R
\]
has infinite rank.

\subsection{Construction of infection knots}

In this section we will prove
Lemma~\ref{lemma:knots-with-independent-signatures}.  For this
purpose, we need the following known facts.  The first is a formula
for the signature of cable knots.

\begin{lemma}
  [Repametrization formula
  \cite{Litherland:1979-1,Cochran-Orr:1993-1,Cha-Ko:2000-1,Cha:2003-1}]
  \label{lemma:reparametrization-formula}
  For a knot $K$, let $K'$ be the $(r,1)$-cable of~$K$.  Then
  $\sigma_{K'}(\omega) = \sigma_{K}(\omega^r)$ for any $\omega\in
  S^1$.
\end{lemma}

The second is a realization result of a ``bump'' signature function.

\begin{lemma}
  \label{lemma:knot-signature-bump}
  For any $\theta_0\in (0,\pi)$, there is a knot $K$ and an
  arbitrarily small neighborhood $I$ of $\theta_0$ contained in
  $(0,\pi)$ such that $\sigma_K(e^{\theta_0\sqrt{-1}})\ne 0$ and
  $\sigma_K(e^{t\sqrt{-1}})=0$ for $t\in [0,\pi]-I$.
\end{lemma}

\begin{proof}
  In~\cite[Proof of Theorem 1]{Cha-Livingston:2002-1}, the following
  statement was shown: for any given $\theta_0\in (0,\pi)$, there are
  $\theta \in (0,\theta_0)$ arbitrarily close to $\theta_0$ and a knot
  $K$ such that $\sigma_K(e^{t\sqrt{-1}})=0$ for $0\le t < \theta$ and
  $\sigma_K(e^{t\sqrt{-1}})$ is a nonzero constant for $\theta < t \le
  \pi$.  Since the Levine-Tristram signature is additive under
  connected sum, i.e., $\sigma_{K\# J}(\omega) = \sigma_{K}(\omega) +
  \sigma_{J}(\omega)$, our conclusion follows immediate by applying
  the above statement twice.
\end{proof}

\begin{proof}[Proof of Lemma~\ref{lemma:knots-with-independent-signatures}]
  We use the following notations: for $\omega\in S^1$,
  $N_\epsilon(\omega)$ denotes the $\epsilon$-neighborhood of $\omega$
  in $S^1$ (with respect to the arc length metric) and
  $\arg(\omega)\in[0,2\pi)$ denotes the argument of~$\omega$.

  Suppose $\theta_0$ and $\theta_1$ satisfying
  $0<\theta_1<\theta_0<\pi/2$ are given.  Let $\omega_0$ and
  $\omega_1$ be points in $S^1$ such that $\theta_i=\arg(\omega_i)$.
  We claim that there is a knot $K$ with the following properties:
  \begin{enumerate}
  \item $\sigma_K(\omega_1) \ne 0$.
  \item $\sigma_K(\omega)=0$ whenever $\omega$, $-\omega$,
    $\bar\omega$, and $-\bar\omega$ are not in~$I$, where
    \[
    I=\Big\{\omega \,\Big|\, \frac{\theta_1}{3} < \arg(\omega) <
    \theta_0\Big\}.
    \]
  \item $\int_{S^1} \sigma_{K}(\omega) \, d\omega = 0$.
  \end{enumerate}
  To prove the claim, choose $\epsilon>0$ such that
  $\epsilon<\min\{\theta_0-\theta_1,\theta_1/3\}$.  By
  Lemma~\ref{lemma:knot-signature-bump}, there is a knot $J$ such that
  $\sigma_J(\omega_1)>0$ and $\sigma_J(\omega)=0$ for $\omega\in
  S^1-[N_\epsilon(\omega_1)\cup N_\epsilon(\bar\omega_1)]$.  Let $J'$
  be the $(2,1)$-cable of~$J$.  Then by
  Lemma~\ref{lemma:reparametrization-formula},
  $\sigma_{J'}(\omega)=\sigma_J(\omega^2)$.  It follows that
  $\sigma_{J'}(\omega) =0$ for
  \[
  \omega\in S^1 - \big[ N_{\epsilon/2}(\sqrt{\omega_1})\cup
  N_{\epsilon/2}(-\sqrt{\omega_1})\cup
  N_{\epsilon/2}(\overline{\sqrt{\omega_1}})\cup
  N_{\epsilon/2}(-\overline{\sqrt{\omega_1}})\big]
  \]
  where $\sqrt{\omega_1}=e^{\theta_1\sqrt{-1}/2}$.  Let $K=J\# -J'$.
  Then, since $\epsilon/2 < \theta_1/2$, $\sigma_{J'}(\omega_1)=0$ and
  so $\sigma_K(\omega_1)=\sigma_J(\omega_1)>0$.  From this (1) follows.
  Since
  \[
  \theta_1/3 < \theta_1/2-\epsilon/2 < \theta_1+\epsilon < \theta_0,
  \]
  both $N_\epsilon(\omega_1)$ and $N_{\epsilon/2}(\sqrt{\omega_1})$
  are contained in~$I$.  From this (2) follows.  Since $\int_{S^1}
  \sigma_{J'}(\omega) \, d\omega = \int_{S^1} \sigma_{J}(\omega) \,
  d\omega$, (3) follows.

  Now, choose powers $d_1,d_2,\ldots$ of $p$ such that $d_1\ge 4$ and
  $d_{i+1}>3d_i$.  For each~$i$, applying the above claim to
  $\theta_0=2\pi/3d_{i-1}$ and $\theta_1=2\pi/d_i$, choose inductively
  a knot $K_i$ which satisfies the above (1), (2), and~(3).  Replacing
  $K_i$ by $K_i \# K_i$ if necessary, we may assume that $K_i$ has
  vanishing Arf invariant.  Then it can be checked easily that
  $\{d_i\}$ and $\{K_i\}$ satisfy
  Lemma~\ref{lemma:knots-with-independent-signatures}.
%
%   by appealing to the above claim repeatedly, for $i=1,2,\ldots$
%   we can inductively choose a power $d_i\ge 4$ of $p$, a knot $K_i$,
%   and a connected neighborhood $I_i$ of $\zeta_{d_i}$ in the upper
%   hemisphere $\{z\in S^1 \mid \Im z > 0 \}$ which have the following
%   properties:
%   \begin{enumerate}
%   \item $d_{i+1}>2d_i$.
%   \item $\sigma_{K_i}(\zeta_{d_i}) > 0$.
%   \item $\sigma_{K_i}(\omega)=0$ whenever $\omega$, $-\omega$,
%     $\bar\omega$, and $-\bar\omega$ are not in~$I_i$.
%   \item $I_i$ and $I_j$ are disjoint for $i\ne j$.
%   \end{enumerate}
%   Furthermore, replacing $K_i$ by $K_i \# K_i$ if necessary, we may
%   assume that $K_i$ has vanishing Arf invariant.  It follows
%   immediately that $\{d_i\}$ and $\{K_i\}$ satisfy
%   Lemma~\ref{lemma:knots-with-independent-signatures}.
%
\end{proof}

\subsection{Construction of an infection curve and an associated
  $p$-structure}

In this subsection we will prove
Lemma~\ref{lemma:infection-curve-and-tower}.  Observe that we may
assume that $m=2$, i.e., $X=S^1\vee S^1$; once this special case is
proved, the general case of $m>2$ follows easily by attaching $m-2$
additional circles.

We begin with a description of a loop $\alpha$ in $X$ (which is
determined by~$n$).  Define loops $\alpha_k$ and $\beta_k$ in $X$
inductively as follows: let $\alpha_0=x_0$ and $\beta_0=x_1$, i.e.,
the paths representing the two (oriented) 1-cells of $X$, and
\[
\alpha_{k+1}=(\alpha_k,\beta_k)=\alpha_k^{\vphantom{-1}}
\beta_k^{\vphantom{-1}} \alpha_k^{-1} \beta_k^{-1}, \quad \beta_{k+1}
= \alpha_k^{\vphantom{-1}}
(\alpha_k^{\vphantom{-1}},\beta_k^{\vphantom{-1}}) \alpha_k^{-1}.
\]
Note that $\alpha_n, \beta_n\in \pi_1(X)^{(n)}$.  Let $\alpha=\alpha_n$.

Next, we construct a $p$-tower $\{X_k\}$ of height $n$ and investigate
the behaviour of the lifts $\tilde\alpha_j$ of $\alpha$ in $X_n$,
similarly to \cite[Section~7.1]{Cha:2007-1}.  Fix a power $q>2$ of
$p$, and let $\Gamma=\Z_q\oplus \Z_q$.  Let $X_0=X$, and let $c_0$,
$d_0$ be the two (oriented) 1-cells of $X_0$ corresponding to $\alpha_0$
and $\beta_0$, respectively.

Suppose a cover $X_k$ of $X$ has been defined and two 1-cells $c_k$
and $d_k$ of $X_k$ has been chosen.  Define a map $X_k\to K(\Gamma,1)$
by sending all $0$-cells and $1$-cells of $X_k$ except $c_k$, $d_k$ to
a basepoint $*\in K(\Gamma,1)$ and sending $c_k$, $d_k$ to loops
representing $(1,0)$, $(0,1) \in \Gamma$, respectively, and let
$\phi_k\colon \pi_1(X_k) \to \Gamma$ be the map induced by this.  We
define $X_{k+1}$ to be the regular cover of $X_k$ determined
by~$\phi_k$.

$X_{k+1}$~can also be described by a cut-paste construction: let $Y_k$
be $X_k$ with $c_k$ and $d_k$ removed, and let $Y(g)$ be a copy of $Y$
for $g\in \Gamma$.  Then
\[
X_{k+1} = \bigg(\bigcup_{g\in\Gamma} Y_k(g)\bigg) \cup \{ \text{1-cells
  $c_k(g)$, $d_k(g)$} \}_{g\in\Gamma}
\]
where $c_k(g)$ goes from (starting point of $c_k) \in Y_k(g)=Y$
to (endpoint of $c_k) \in Y_k(g+(1,0)) = Y_k$, and $d_k(g)$
goes from (starting point of $d_k) \in Y_k(g)=Y$ to (endpoint of $d_k)
\in Y_k(g+(0,1)) = Y_k$.  See Figure~\ref{figure:cover-construction}.
Define $c_{k+1}=c_k(0,0)$ and $d_{k+1}=-c_k(1,1)$, i.e.,
$c_k(1,1)$ with reversed orientation, as illustrated in
Figure~\ref{figure:cover-construction}.

\begin{figure}[ht]
  \begin{center}
    \includegraphics[scale=.8]{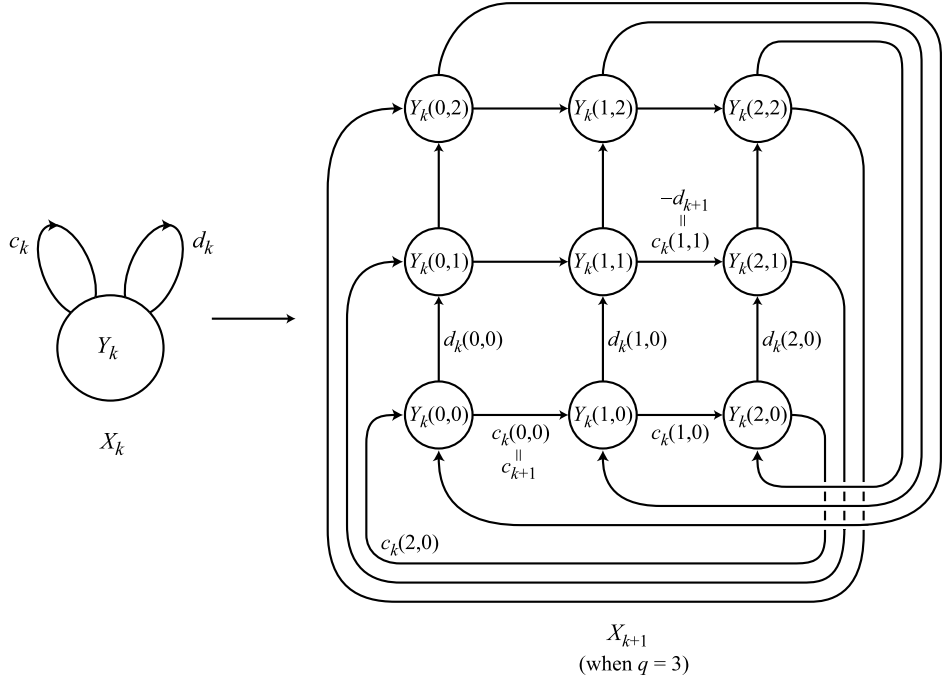}
  \end{center}
  \caption{Construction of $X_{k+1}$ from $X_k$}
  \label{figure:cover-construction}
\end{figure}

Let $\bar X_{k+1}$ be the 1-complex obtained from $X_{n+1}$ by
collapsing each $Y_n(g)$ ($g\in\Gamma$) to a point.  For a path
$\gamma$ in $X_{k+1}$, we denote the image of $\gamma$ in $\bar
X_{k+1}$ by~$\bar\gamma$.  In particular the image of the 1-cells
$c_k(g), d_k(g)$ are denoted by $\bar c_k(g), \bar d_k(g)$

We choose a basepoint $*\in X_k$ as follows: $*\in X_0 = S^1\vee S^1$ is
the wedge point, and $*\in X_{k+1}$ is defined to be the pre-image of
$*\in X_k$ contained in $Y_k(0,0) \subset X_{k+1}$.  Also, we sometimes
regard $* \in X_k$ as a point in~$Y_k(g)$.

\begin{lemma}
  \label{lemma:lift-behaviour}
  For a $0$-cell $v$ in $X_{k+1}$, let $\gamma_v$ and $\delta_v$ be
  the lift of $\alpha_{k+1}$ and $\beta_{k+1}$ in $X_{k+1}$ based at $v$,
  respectively.  Note that $v\in Y_k(g) \subset X_{k+1}$ for some $g\in\Gamma$.
  \begin{enumerate}
  \item If $v \ne *\in Y_k(g)$, then $\bar\gamma_v$ and $\bar\delta_v$
    are null-homotopic (rel $\partial$) in~$\bar X_{k+1}$.
  \item If $v=* \in Y_k(g)$, then $\bar\gamma_v$ and $\bar\delta_v$
    are homotopic (rel $\partial$) to
    \begin{gather*}
      \bar c_k(g) \bar d_k(g+(1,0)) \bar c_k(g+(0,1))^{-1}
      \bar d_k(g)^{-1} \text{ and}\\
      \bar c_k(g) \bar c_k(g+(1,0)) \bar d_k(g+(2,0)) \bar c_k(g+(1,1))^{-1}
      \bar d_k(g+(1,0))^{-1} \bar c_k(g)^{-1}
    \end{gather*}
    in $\bar X_{k+1}$, respectively.
  \end{enumerate}
\end{lemma}

\begin{proof}
  We use an induction on~$k$. When $k=1$, the conclusion is easily
  verified.  Suppose it holds for~$k$.  Let $v' \in X_k$ be the image
  of $v$ under $X_{k+1} \to X_k$.  Let $a$ and $b$ be the lifts of
  $\alpha_k$ and $\beta_k$ in $X_k$ based at~$v'$.  Since $\alpha_{k+1} =
  \alpha_k^{\vphantom{-1}} \beta_k^{\vphantom{-1}} \alpha_k^{-1} \beta_k^{-1}$,
  $\gamma_v$~is obtained by concatenating appropriate lifts of $a$,
  $b$, $a^{-1}$, and $b^{-1}$ in~$X_{k+1}$.

  If $\bar a$ is homotopic (rel $\partial$) to a path which does not
  pass through $\bar c_k$ nor $\bar d_k$ in $\bar X_k$, then by the
  construction of $X_{k+1}$ from $X_k$, any lift of $a$ in $X_{k+1}$
  is contained in some $Y_k(-)$, and so $\bar\gamma$ is null-homotopic
  (rel $\partial$) in~$\bar X_{k+1}$.  The analogue for $b$ holds too.
  From this it follows that $\bar\gamma_v$ can be non-null-homotopic
  only if any paths homotopic to $\bar a$ or $\bar b$ pass through
  either $\bar c_k$ or~$\bar d_k$.  By our conclusion for $k$, this
  can be satisfied only when $v'=* \in X_k$.  Furthermore, if this is
  the case, it can be verified that the image in $\bar X_{k+1}$ of
  lifts of $a$, $b$ in $X_{k+1}$ are of the form $\bar c_n(-)$, $\bar
  d_n(-)$, respectively, so that $\bar\gamma_v$ is of the desired
  form.  The conclusion for $\bar\delta_v$ is proved similarly.
\end{proof}

\begin{proof}[Proof of Lemma~\ref{lemma:infection-curve-and-tower}]
  We will show that our $\alpha$ and $\{X_k\}$ satisfy the conclusion
  of Lemma~\ref{lemma:infection-curve-and-tower}.  The remaining thing
  to show is that there is $f\colon \pi_1(X_n) \to \Z$ with the
  properties described in
  Lemma~\ref{lemma:infection-curve-and-tower}~(2).  Choose a map $\bar
  X_n \to S^1$ which sends all 0-cells of $\bar X_n$ to a fixed
  basepoint in $S^1$, sends the 1-cells $\bar c_n=\bar c_n(0,0)$ and
  $\bar c_n(1,0)$ to a loop generating $\pi_1(S^1)=\Z$ and its
  inverse, respectively, and sends all other 1-cells to the basepoint.
  Let $f\colon \pi_1(X_n) \to \Z$ be the map induced by $X_n \to \bar
  X_n \to S^1$.

  Then, by using Lemma~\ref{lemma:lift-behaviour}, it is easily
  verified that $f$ sends all the lifts of $\alpha$ to either $-1$,
  $0$, or $1$; for, if the image $\bar\gamma$ of a lift of $\alpha$ in
  $\bar X_n$ is not null-homotopic, then from
  Lemma~\ref{lemma:lift-behaviour} it follows that $\bar\gamma$ should
  be of the form described in Lemma~\ref{lemma:lift-behaviour}~(2),
  which can pass through $\bar c_{n-1}(0,0)$ and $\bar c_{n-1}(1,0)$
  at most once but cannot pass through both.  Also, by
  Lemma~\ref{lemma:lift-behaviour}~(2), (the image of) the lift of
  $\alpha$ based at $*\in X_n$ passes through $\bar c_{n-1}(0,0)$
  exactly once but never passes through $\bar c_{n-1}(1,0)$.  So, this
  lift is sent to $1$ by~$f$.

  If $\gamma$ is a loop in $X_n$ which projects to a power of $\alpha_i$ in
  $X$, then it can be seen that $\bar\gamma$ in $\bar X_n$ is either
  null-homotopic or of the form $\big(\prod_i \bar
  c_n(g+(i,0))\big)^a$ or $\big(\prod_j \bar d_n(g+(0,j))\big)^a$.
  From the definition of $f$, it follows that $f$ sends $\gamma$
  to~$0$.  This shows the local triviality claim (see
  Remark~\ref{remark:alternative-description-of-local-triviality}).
\end{proof}

\section{Independence of links}
\label{section:independence-of-links}

Recall that a connected sum of two links $L_1$ and $L_2$ is defined by
choosing disk basings; given a disk basing of each $L_i$, we obtain a
string link $\beta_i$ whose closure is $L_i$, and the closure of
$\beta_1\beta_2$ is defined to be a connected sum of $L_1$ and $L_2$.
In this section we investigate ``independence'' of links under
connected sum.

For this purpose, we need a result due to Habegger and
Lin~\cite{Habegger-Lin:1998-1}.  Following
\cite{Habegger-Lin:1990-1,Habegger-Lin:1998-1}, we define a left
action $\Sigma\cdot\colon \beta \to \Sigma\beta$ of a $2m$-string link
$\Sigma$ on an $m$-string link $\beta$ as in
Figure~\ref{figure:habegger-lin-action}.  The right action $\cdot
\Sigma\colon \beta \to \beta\Sigma$ is defined similarly.  Denote
$\mathcal{S}=\{\Sigma \mid 1_m\Sigma=1_m\}$ where $1_m$ denotes the
trivial $m$-string link.

\begin{figure}[ht]
  \begin{center}
    \includegraphics[scale=.9]{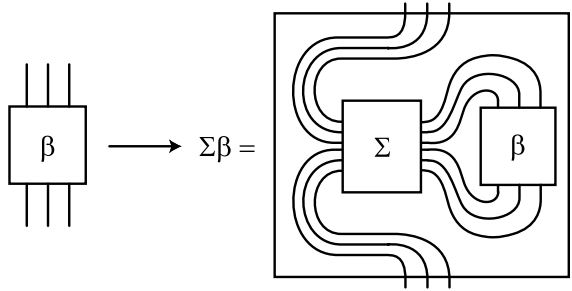}
    \caption{Left action of $\Sigma$ on $\beta$}
    \label{figure:habegger-lin-action}
  \end{center}
\end{figure}

\begin{proposition}[\cite{Habegger-Lin:1998-1}]
  \label{proposition:habegger-lin-action}
  The closures of two $m$-string links $\beta_1$ and $\beta_2$ are
  concordant (as links) if and only if $\beta_2$ is concordant to
  $\Sigma\beta_1$ (as string links) for some $2m$-string link $\Sigma
  \in \mathcal{S}$.
\end{proposition}

Throughout this section, we assume $p=2$.  Fix $m>1$ and $n$, and
consider the string links $\beta(K_i)$ constructed in
Section~\ref{section:examples} (using $p=2$).  Namely, $K_i$ is as in
Lemma~\ref{lemma:knots-with-independent-signatures} and $\beta(K_i)$
is obtained from the trivial string link by infection using $K_i$
along a curve realizing the homotopy class of the loop given by
Lemma~\ref{lemma:infection-curve-and-tower}.  Let $L_i$ be the closure
of $\beta(K_i)$.

\begin{theorem}
  \label{theorem:independence-of-links-over-Z}
  Suppose that a connected sum of the $a_iL_i$ and some local knots is
  $\Z_{(p)}$-coefficient $(n.5)$-solvable for some integers $a_i$,
  some disk basings, and some order of the~$a_iL_i$.  Then $a_i=0$ for
  all~$i$.
\end{theorem}

As mentioned in Section~\ref{section:examples}, each $\beta(K_i)$ is
(integrally) $(n)$-solvable and of vanishing Harvey's
$\rho_n$-invariant, and thus so is~$L_i$.  Therefore, from
Theorem~\ref{theorem:independence-of-links-over-Z},
Theorem~\ref{theorem:independent-links-over-Z} follows immediately:
the links $L_i$ are $(n)$-solvable, has vanishing $\rho_n$-invariants,
and independent over $\Z$ modulo $\FL_{(n.5)}$ and local knots.

\begin{proof}
  We may assume that $a_i \ge 0$ by replacing $L_i$ by $-L_i$ (and
  $K_i$ by $-K_i$) if necessary.  First we consider the special case
  that there is no local knot summand; suppose $\beta_{ij}'$ ($1\le j
  \le a_i$) is a string link with closure $L_i$ and the product
  $\beta'$ of the $\beta_{ij}'$ (in some order) is
  $\Z_{(p)}$-coefficient $(n.5)$-solvable.  By
  Proposition~\ref{proposition:habegger-lin-action}, there is
  $\Sigma_{ij} \in \mathcal{S}$ such that $\beta_{ij}'$ is concordant
  to $\Sigma_{ij} \beta(K_i)$.  Let $\beta_{ij}=\Sigma_{ij}
  \beta(K_i)$.  Then the product
  \[
  \beta = \prod_{i,j} \beta_{ij}
  \]
  is $\Z_{(p)}$-coefficient $(n.5)$-solvable (for some order of the
  factors) since $\beta$ is concordant to~$\beta'$.

  Suppose not all $a_i$ are zero.  Choose the minimal $i_0$ such that
  $a_{i_0}\ne 0$.  We will derive a contradiction by showing that
  $a_{i_0}$ should be zero.  For this purpose, we need the following
  facts: let $\{X_k\}$ be the $p$-tower constructed in
  Lemma~\ref{lemma:infection-curve-and-tower}.  Then

  \begin{enumerate}
  \item For each $i$, there is a character $\theta_i\colon\pi_1(X_n)
    \to \Z_{d_i}$ such that the $p$-structure
    $\T_i=(\{X_k\},\theta_i)$ is locally trivial and satisfies
    $\lambda_{\T_i}(\beta(K_i))=c\cdot \sigma_{K_i}(\zeta_{d_i})$ for
    some constant $c>0$.
  \item For any $p$-structure $\T=(\{X_k\},\theta)$ of height $n$ with
    $\theta\colon \pi_1(X_n) \to \Z_d$, we have
    $\lambda_{\T}(\beta(K)) = \sum_k \sigma_{K}(\zeta_{d}^{s_k})$ for
    some $\{s_k\}$.
  \end{enumerate}

  (1) was proved in the proof of
  Theorem~\ref{theorem:Z-independence-of-string-links} using
  Lemma~\ref{lemma:infection-curve-and-tower}.  Observing that our
  infection curve $\alpha$ producing $\beta(K_i)$ lifts to the $n$th
  term of the $p$-tower, (2)~follows immediately from
  Proposition~\ref{proposition:lambda-of-infected-manifold}, as in the
  proof of Theorem~\ref{theorem:Z-independence-of-string-links}.

  Among the factors $\beta_{i_0 j}$ ($1\le j\le a_{i_0}$), choose an
  arbitrary one, say~$\beta_{i_0 j_0}$.  Appealing to (1) above,
  choose a $p$-structure $\T_{i_0 j_0}$ of height $n$ such that
  $\lambda_{\T_{i_0 j_0}}(\beta(K_{i_0}))=c\cdot
  \sigma_{K_{i_0}}(\zeta_{d_{i_0}})$.

  Let $\mu_{ij}$ be the composition of the preferred meridian map $X
  \to E_{\beta(K_i)}$ for $\beta(K_i)$ and the natural inclusion
  \[
  E_{\beta(K_i)} \to E_{\beta_{ij}} \to E_\beta \to M_\beta
  \]
  and $\mu\colon X\to M_\beta$ be the preferred meridian map
  for~$\beta$.  Then, since $\beta$ is an $\hF$-string link, $\mu$ and
  the $\mu_{ij}$ are $p$-tower maps by~\cite[Proposition
  6.3]{Cha:2007-1}.  So the $p$-structure $\T_{i_0j_0}$ determines a
  $p$-structure $\T'$ for $M_\beta$ via $\mu_{i_0j_0}$, and then $\T'$
  induces a $p$-structure $\T_{ij}$ for $X$ via $\mu_{ij}$ for each
  $i,j$.  Also, $\T'$ induces a $p$-structure $\T$ for $X$ via~$\mu$.

  From our choice of the $p$-structures, it follows that
  \[
  \lambda_\T(\beta) = \sum_{i,j} \lambda_T(\beta_{ij}) = \sum_{i,j}
  \lambda_{\T_{ij}} (\beta(K_i)).
  \]
  By the property stated in
  Lemma~\ref{lemma:knots-with-independent-signatures}~(1), we have
  \[
  \lambda_{T_{i_0 j_0}}(\beta(K_{i_0})) = c \sigma_{K_{i_0}}(\zeta_{d_{i_0}}) \ne 0.
  \]
  For $i=i_0$ and $j\ne j_0$, we have
  \[
  \lambda_{\T_{i_0 j}}(\beta(K_{i_0})) = \sum_k
  \sigma_{K_{i_0}}(\zeta_{d_{i_0}}^{s_k})
  \]
  for some $\{s_k\}$ by (2) above.  By the property stated in
  Lemma~\ref{lemma:knots-with-independent-signatures}~(1), each
  summand of $\lambda_{\T_{i_0 j}}(\beta(K_{i_0}))$ is either zero or
  of the same sign with $\lambda_{T_{i_0 j_0}}(\beta(K_{i_0}))$ (Here
  we need the assumption that $p=2$).  For $i < i_0$, there is no
  $(i,j)$-summand in the above expression of $\lambda_\T(\beta)$ since
  $a_i=0$ by our choice of~$i_0$.  For $i>i_0$, for some $\{s_k\}$ we
  have
  \[
  \lambda_{\T_{i j}}(\beta(K_{i})) = \sum_k
  \sigma_{K_{i}}(\zeta_{d_{i_0}}^{s_k}) = 0
  \]
  by the property stated in
  Lemma~\ref{lemma:knots-with-independent-signatures}~(2).

  So it follows that $\lambda_\T(\beta)\ne 0$.  This contradicts that
  $\beta$ is $\Z_{(p)}$-coefficient $(n.5)$-solvable, by
  Theorem~\ref{theorem:COT-filtration-and-lambda_T}.  It completes the
  proof when there is no local knot summand.

  For the general case, suppose that the product of $\beta$ and some
  local knots, say $\beta_k$, is $\Z_{(p)}$-coefficient
  $(n.5)$-solvable.  In this case, we need an additional argument as
  described below.  Observe that now we have
  \[
  0 = \sum_{i,j} \lambda_{\T_{ij}} (\beta(K_i)) + \sum_k
  \lambda_\T(\beta_k).
  \]
  So it suffices to show that $\lambda_\T(\beta)$ vanishes for any
  local knot $\beta$.  Let denote $\T'=(\{M_k\},\phi)$.  Since
  $\T_{i_0 j_0}=(\{X_k\},\theta_{i_0})$ is locally trivial and induced
  by $\T'$ via the meridian map $\mu_{i_0 j_0}$, any loop in $M_n$
  that projects to a power of a meridian in $M_\beta$ is in the kernel
  of~$\phi$, by Lemma~\ref{lemma:local-triviality-under-pullback} (see
  also
  Remark~\ref{remark:alternative-description-of-local-triviality}).
  Applying Lemma~\ref{lemma:local-triviality-under-pullback} to the
  preferred meridian map $\mu$, it follows that $\T$ is locally
  trivial.  Therefore $\lambda_\T(\beta)=0$ for any local knot
  $\beta$, by Theorem~\ref{theorem:vanishing-for-local-knots}.  This
  finishes the proof.
\end{proof}

The arguments in~\cite[Section~7]{Cha:2007-1} (in particular see
Lemma~7.7 of~\cite{Cha:2007-1}) prove that the $n$th iterated Bing
double $BD_n(K)$ of a knot $K$ is the closure of a string link
obtained from the trivial string link by infection along a curve
$\alpha$, which satisfies the conclusion of our
Lemma~\ref{lemma:infection-curve-and-tower}; while the local
triviality condition is not mentioned in~\cite{Cha:2007-1}, it can be
satisfied by a minor change of the construction of the character
in~\cite{Cha:2007-1}.  So, our argument shows the following:

\begin{theorem}
  Suppose $\{K_i\}$ is a family of knots satisfying
  Lemma~\ref{lemma:knots-with-independent-signatures}.  Then, for any
  $n$, the links $BD_n(K_i)$ are $(n)$-solvable, have vanishing
  $\rho_n$-invariant, and are independent over $\Z$ modulo
  $\FL_{(n.5)}$ and local knots.
\end{theorem}

Using a similar technique, we can also show the independence of
certain 2-torsion iterated Bing doubles considered
in~\cite{Cha:2007-1}, namely we can show
Theorem~\ref{theorem:independent-links-over-Z2}: there are infinitely
many amphichiral knots $K_{i}$ such that for any $n$, the links
$BD_n(K_{i})$ are 2-torsion, $(n)$-solvable, and independent over
$\Z_2$ modulo $\FL_{(n+1.5)}$ and local knots.

\begin{proof}[Proof of Theorem~\ref{theorem:independent-links-over-Z2}]
  In this proof we use tools from algebraic number theory: the
  discriminant map
  \[
  \dis\colon L^0(\Q(\zeta_d)) \to
  \frac{\Q(\zeta_d^{\vphantom{-1}}+\zeta_d^{-1})^\times}{\{z\bar z
    \mid z\in \Q(\zeta_d)^\times\}}
  \]
  and the norm residue symbol $(x,y)_q \in \{\pm 1\}$ which is defined
  for $x,y\in \Q(\zeta_d^{\vphantom{-1}}+\zeta_d^{-1})^\times$ and for
  a prime~$q$ of $\Q(\zeta_d^{\vphantom{-1}}+\zeta_d^{-1})$.  (Here
  $\{\pm 1\}$ is regarded as a multiplicative group of order two.)
  Essentially what we need is the following fact: $\dis$ is a group
  homomorphism, and when $d=4$ (i.e., $\zeta_d=\sqrt{-1}$),
  $(x,-1)_q=1$ for all prime $q$ if $x\in \Q^\times$ is of the form
  $z\bar z$ for some $z\in \Q(\sqrt{-1})^\times$.  Interested readers
  who are not familiar with these tools are referred to Section~4.5 of
  \cite{Cha:2007-1} and Section 3.4 of~\cite{Cha:2003-1}, which
  provide more detailed accounts on the discriminant and norm residue
  symbols for non-experts of algebraic number theory.
  
  We will show that the family $\{K_i\}$ of knots constructed in the
  proof of Corollary~8.5~(1) in~\cite{Cha:2007-1} satisfies our
  conclusion.  It was shown that $BD_n(K_i)$ is 2-torsion and
  $(n)$-solvable in~\cite{Cha:2007-1}.  The properties stated below,
  which are shown by the arguments of the proof of Corollary~8.5
  of~\cite{Cha:2007-1} (see also Proposition~5.6
  of~\cite{Cha:2007-1}), are the only facts on the $K_i$ that we need
  in order to prove the independence: $BD_n(K_i)$ is the closure of a
  string link $\beta_i$ which admits a ``dual prime'' $p_i$ satisfying
  the following:
  \begin{enumerate}
  \item For each $i$, there is a locally trivial $\Z_4$-valued
    $2$-structure $\T_i$ of height $n+1$ such that
    \[
    \big( \dis\lambda_{\T_i}(\beta_i),-1 \big)_{p_i} = -1.
    \]
  \item If $i\ne j$, then for any $\Z_4$-valued $2$-structure $\T$ of
    height $n+1$,
    \[
    \big( \dis\lambda_{\T}(\beta_j),-1 \big)_{p_i} = +1.
    \]
  \end{enumerate}
  (Note that $\lambda_{\T}(-)$ and $\lambda_{\T_i}(-)$ are always in
  $L^0(\Q(\sqrt{-1}))$ so that the discriminant is in $\Q$, since the
  $p$-structures are $\Z_4$-valued.)

  Now suppose that for some finite nonempty subset $I$ of
  $\mathbb{N}$, a connected sum of $\{L_{i}\}_{i\in I}$ and some local
  knots is $\Z_{(p)}$-coefficient $(n+1.5)$-solvable.  Choose any $i_0
  \in I$.  Then, by appealing to Habegger-Lin's
  Proposition~\ref{proposition:habegger-lin-action} and by choosing
  appropriate locally trivial $p$-structures via meridian maps as in
  the proof of Theorem~\ref{theorem:independence-of-links-over-Z}, we
  can see that there are $p$-structures $\T_i$ such that 
  \[
  0 = \sum_{i\in I} \lambda_{\T_{i}}(\beta_{i})
  \]
  where $\T_{i_0}$ satisfies $(\dis(\lambda_{T_{i_0}}(\beta_{i_0}),-1
  \big)_{p_{i_0}}=-1$ as in~(1).  Taking the discriminant and
  evaluating the norm residue symbol $(\,\cdot\,, -1)_{p_{i_0}}$, we
  have
  \[
  1=\prod_{i\in I} \big(\dis(\lambda_{T_i}(\beta_i),-1
  \big)_{p_{i_0}}.
  \]
  But, by (2) we have $(\dis(\lambda_{T_i}(\beta_i),-1
  \big)_{p_{i_0}}=1$ for $i\ne i_0$.  This is a contradiction.
\end{proof}

\bibliographystyle{amsplainabbrv}
\renewcommand{\MR}[1]{}

\bibliography{research}

\end{document}